\documentclass{amsart}
\usepackage{amssymb}
\usepackage[all]{xy}
\newcommand\datver[1]{\def\datverp%
 {\par\boxed{\boxed{\text{#1; Run: \today}}}}}

\newcommand\boxb[1]{\square_b}

\numberwithin{equation}{section}

\newcommand\paperbody%
        {}

\newcommand\Ko{\operatorname{K^1}}
\newcommand\Ke{\operatorname{K^0}}
\newcommand\Kco{\operatorname{K_c^1}}
\newcommand\Kce{\operatorname{K_c^0}}

\newcommand\sH{\operatorname{H^*}}

\newcommand\ev{\operatorname{ev}}
\newcommand\tev{\operatorname{\widetilde{ev}}}
\newcommand\Ev{\operatorname{Ev}}
\newcommand\tEv{\operatorname{\widetilde{Ev}}}

\newcommand\ps{\operatorname{ps}}

\newtheorem{lemma}{Lemma}
\newtheorem{proposition}{Proposition}
\newtheorem{corollary}{Corollary}
\newtheorem{theorem}{Theorem}
\newtheorem{non-theorem}{Non-Theorem}
\newtheorem{conjecture}{Conjecture}

\theoremstyle{remark}
\newtheorem{definition}{Definition}
\newtheorem{remark}{Remark}

\newcommand\ad{\operatorname{ad}}

\newcommand\coF{{}^{\mathcal{C}}\kern-2pt\Lambda}

\newcommand\cFTs{{}^{\Phi}\overline{T}\kern-1pt{}^*}

\newcommand\sus{\operatorname{sus}}
\newcommand\psus{\operatorname{ps}}

\newcommand\even{\text{even}}
\newcommand\odd{\text{odd}}

\newcommand\BC{\operatorname{BC}}

\newcommand\Tr{\operatorname{Tr}}

\newcommand\STr{\operatorname{STr}}

\newcommand\Ch{\operatorname{Ch}}

\newcommand\SA{\operatorname{SA}}
\newcommand\cA{\mathcal{A}}

\newcommand\cB{\mathcal{B}}

\newcommand\cD{\mathcal{D}}

\newcommand\cF{\mathcal{F}}
\newcommand\cG{\mathcal{G}}

\newcommand\cL{\mathcal{L}}

\newcommand\bbE{\mathbb E}

\newcommand\bbB{\mathbb B}
\newcommand\bbC{\mathbb C}

\newcommand\bbN{\mathbb N}

\newcommand\bbR{\mathbb R}

\newcommand\bbZ{\mathbb Z}

\newcommand\cS{\mathcal S}
\newcommand\cSp{{\mathcal S}'}

\newcommand\CI{{\mathcal{C}}^{\infty}}

\newcommand\CmI{{\mathcal{C}}^{-\infty}}

\newcommand\Diff[1]{\operatorname{Diff}^{#1}}

\newcommand\cFNs{{}^{\Phi}\overline N\kern-1pt{}^*}

\newcommand\ind{\operatorname{ind}}

\newcommand\Hom{\operatorname{Hom}}
\newcommand\Id{\operatorname{Id}}

\newcommand\SU{\operatorname{SU}}

\newcommand\dCI{\dot{\mathcal{C}}^{\infty}}

\newcommand\ha{\frac{1}{2}}

\newcommand\sign{\operatorname{sign}}

\newcommand\pa{\partial}

\newcommand\dR{\operatorname{dR}}

\newcommand\pr{\operatorname{pr}}
\newcommand\U{\operatorname{U}}
\newcommand\cli{\bbC\ell}
\newcommand\End{\operatorname{End}}

\newcommand\tcA{\widetilde{\cA}}

\newcommand\cU{\mathcal{U}}

\newcommand\Mand{\text{ and }}

\newcommand\Mforsome{\text{ for some }}

\newcommand\Mon{\text{ on }}
\newcommand\Mor{\text{ or }}
\newcommand\Mover{\text{ over }}
\newcommand\Msatisfies{\text{ satisfies }}

\newcommand\Mwhere{\text{ where }}

\newcommand\cf{cf\@. }
\datver{1.0E; Revised: 1-5-2009}
\begin{document}
\title[Eta forms and families index]
{Eta forms and the odd pseudodifferential families index}

\author{Richard Melrose}
\address{Department of Mathematics, Massachusetts Institute of Technology}
\email{rbm@math.mit.edu}
\author{Fr\'ed\'eric Rochon}
\address{Department of Mathematics, Australian National University}
\email{frederic.rochon@anu.edu.au}

\thanks{The research of the first author was partially supported by the
  National Science Foundation under grant DMS-0408993; the second author
  was supported by a NSERC discovery grant.}
%
\begin{abstract}
Let $A(t)$ be an elliptic, product-type suspended (which is to say
parameter-dependant in a symbolic way) family of pseudodifferential
operators on the fibres of a fibration $\phi$ with base $Y.$ The standard
example is $A+it$ where $A$ is a family, in the usual sense, of first
order, self-adjoint and elliptic pseudodifferential operators and
$t\in\bbR$ is the `suspending' parameter. Let
$\pi_{\cA}:\cA(\phi)\longrightarrow Y$ be the infinite-dimensional bundle
with fibre at $y\in Y$ consisting of the Schwartz-smoothing perturbations,
$q,$ making $A_y(t)+q(t)$ invertible for all $t\in\bbR.$ The total eta
form, $\eta_{\cA},$ as described here, is an even form on $\cA(\phi)$ which
has basic differential which is an explicit representative of the odd Chern
character of the index of the family:
\begin{equation}
d\eta_{\cA}=\pi_{\cA}^*\gamma _A,\ \Ch(\ind(A))=[\gamma_{A}]\in H^{\odd}(Y).
\tag{*}\label{efatoi.5}\end{equation}
The $1$-form part of this identity may be interpreted in terms of the
$\tau$ invariant (exponentiated eta invariant) as the determinant of the
family. The $2$-form part of the eta form may be interpreted as a
B-field on the K-theory gerbe for the family $A$ with \eqref{efatoi.5}
giving the `curving' as the $3$-form part of the Chern character of the
index. We also give `universal' versions of these constructions over a
classifying space for odd K-theory.  For Dirac-type operators, we relate $\eta_{\cA}$ with the Bismut-Cheeger eta form.  
\end{abstract}

\maketitle

\section*{Introduction}

Eta forms, starting with the eta invariant itself, appear as the boundary
terms in the index formula for Dirac operators \cite{MR53:1655a},
\cite{MR93k:58211}, \cite{MR91e:58181}, \cite{MR99a:58144},
\cite{MR99a:58145}. One aim of the present paper is to show that, with the
freedom gained by working in the more general context of families of
pseudodifferential operators, these forms appear as universal
transgression, or connection, forms for the cohomology class of the
index. That these forms arise in the treatment of boundary problems
corresponds to the fact that boundary conditions amount to the explicit
inversion of a suspended (or model) problem on the boundary. The odd index
of the boundary family is trivial and the eta form is an explicit
trivialization of it in cohomology. To keep the discussion within bounds we
work here primarily in the `odd' setting of a family of self-adjoint
elliptic pseudodifferential operators, taken to be of order $1,$ on the
fibres of a fibration of compact manifolds
\begin{equation}
\xymatrix{Z\ar@{-}[r]&M\ar[d]^{\phi}\\ &Y,}\quad\quad\begin{gathered}\\
\\
A\in\Psi^{1}(M/Y;E),\ A^*=A,\ \text{elliptic}\end{gathered}
\label{efatopi.1}\end{equation}
where a smooth, positive, fibre density on $M$ and a Hermitian inner product
on the bundle have been chosen to define the adjoint. A similar discussion
is possible in the more usual `even' case.

From the fibration and pseudodifferential family an infinite-dimensional
bundle of principal spaces, $\cA(\phi),$ of invertible perturbations on each
fibre, can be constructed:
\begin{equation}
\xymatrix{G^{-\infty}_{\sus}(\phi;E)\ar@{-}[r]\ar[dr]_{q_{\cA}}
&\cA(\phi)\ar[d]_{p_{\cA}}\ar[r]&
\Psi^{1,1}_{\psus}(\phi;E)\\
&
Y.\ar[ur]_{[A+it]}}
\label{efatopi.2}\end{equation}
The vertical map here does not correspond to a principal bundle in the
conventional sense since it has a non-constant bundle of structure groups,
$G^{-\infty}_{\sus}(\phi;E),$ with fibre consisting of the invertible
suspended smoothing perturbations of the identity on the corresponding
fibre of $\phi.$ The individual groups in this bundle are flat pointed loop
groups and hence are classifying for even K-theory. Despite the twisting by
fibre diffeomorphisms, the homotopy group
$\Pi_0(G^{-\infty}_{\sus}(M/Y;E)),$ where $G^{-\infty}_{\sus}(M/Y;E)$ is
the space of global smooth sections of $G^{-\infty}_{\sus}(\phi;E),$ is
canonically identified with $\Ke(Y)$, see for instance \cite{bpffco}. In this sense
$G^{-\infty}_{\sus}(\phi;E)$ is a `classifying bundle' for the K-theory of
$Y.$

Throughout this paper we use notation such as $\cB(\phi)$ for the total
space of a bundle over $Y$ associated with a given fibration
\eqref{efatopi.1} and $\cB(M/Y)$ for the corresponding space of global
sections of the bundle. Thus, on the right in \eqref{efatopi.2},
$\Psi^{1,1}_{\psus}(\phi;E)$ is the space of product-type suspended
pseudodifferential operators on the fibres of $\phi$ (and acting on
sections of the bundle $E)$ -- an element of $\Psi^{1,1}_{\psus}(\phi,E)$
is thus a family of pseudodifferential operators acting 
on smooth sections of $E$
on the fibre above a point $y\in Y$ where the parameter in the family is in
$\bbR$ (the suspension variable $t)$ with `product symbolic' dependence
on this parameter (as indicated by the suffix $\psus).$ In the diagram
above, $A+it$ is such a family, although we consider a more general
situation in the body of the paper. 

On the total space of the structure bundle in \eqref{efatopi.2} there is a
deRham form, $\Ch_{\even},$ representing the even Chern character,
i\@.e\@.~which pulls back under any section to a representative of the
Chern character of the K-class defined by that section. The eta form,
$\eta_{\cA},$ in this setting is a form on $\cA(\phi),$ defined by
regularization of the formula for the Chern character on the structure
bundle (see \eqref{gef.4} and \eqref{efatopi.47}). Under the action of a
section of the structure bundle, this eta form shifts by the pull back of
the Chern character up to an exact term
\begin{equation}
\begin{gathered}
\alpha :G^{-\infty}_{\sus}(M/Y;E)\times_Y\cA(\phi)\longrightarrow\cA(\phi),\ \alpha
^*\eta_{\cA}=\pr_{2}^*\eta_{\cA}+\pr_{1}^*\Ch_{\even}+d\gamma,\\
\gamma\text{ a smooth form on }G^{-\infty}_{\sus}(M/Y;E)\times_Y\cA(\phi),\\
\Mwhere \pr_{1}:G^{-\infty}_{\sus}(M/Y;E)\times_Y\cA(\phi)\longrightarrow 
G^{-\infty}_{\sus}(M/Y;E)\Mand\\
\pr_{2}:G^{-\infty}_{\sus}(M/Y;E)\times_Y\cA(\phi)\longrightarrow\cA(\phi)
\end{gathered}
\label{efatopi.3}\end{equation}
are the natural projections.  The central result below is:

\begin{theorem}\label{efatopi.4} The eta form, $\eta_{\cA},$ on $\cA(\phi)$
  has basic differential representing the odd Chern character of the index
  bundle of the given family $A$ in \eqref{efatopi.1} 
\begin{equation}
d\eta_{\cA}=p_{\cA}^*\gamma_{A} ,\ \gamma_A\in\CI(Y;\Lambda^{\odd}),\
d\gamma_{A}=0,\ \Ch_{\odd}(\ind(A))=[\gamma_{A}]\in H^{\odd}(Y).
\label{efatopi.5}\end{equation}
\end{theorem}
Once a choice of connection is made, the form $\gamma_{\cA}$, which can be
written explicitly in terms of the formal trace of \cite{MR96h:58169} (see
\eqref{gef.14} below), gives a representative of the Chern character of the
index class. For the particular case of families of Dirac
operators associated to a pseudodifferential bundle, Paycha and Mickelsson, in
\cite{Paycha-Mickelsson}, obtained a related representative of the Chern
class using the Wodzicki residue instead of the formal trace.

To prove Theorem~\ref{efatopi.4} we use the smooth delooping sequence for
the fibration, which is the top row in the diagram
\begin{equation}
\xymatrix{G^{-\infty}_{\sus}(\phi;E)\ar[r]\ar@{-}[d]&
\tilde G^{-\infty}_{\sus}(\phi;E)\ar[r]\ar@{-}[d]&
G^{-\infty}(\phi;E)
\\
\cA(\phi)\ar[dr]\ar@{^(->}[r]&
\widetilde{\cA}(\phi)\ar[d]\ar[ur]
\\
&
Y.\ar@/_2ex/[u]_(0.63){\tilde A}\ar[uur]_{\widetilde{\ind}(A)}
}
\label{efatopi.6}\end{equation}
Here $\widetilde{\cA}(\phi)$ is an extension of $\cA(\phi)$ to a bundle of
principal spaces (in the same sense as for $\cA)$ with bundle of structure
groups, $\tilde G^{-\infty}_{\sus}(\phi;E),$ the half-open (smooth-flat)
loop group bundle. This has contractible fibres and hence
$\widetilde{\cA}(\phi)$ has a section $\widetilde{A}$ as indicated in
\eqref{efatopi.6}. Taking the quotient by the original structure group,
this projects to a section, $\widetilde{\ind}(A),$ of $G^{-\infty}(\phi;E)$
with homotopy class giving (minus) the index in $\Ko(Y)$ of the family. Ultimately,
\eqref{efatopi.5} follows from the fact that there is a corresponding
multiplicativity formula linking $\eta_{\cA}$ to an analogous form,
$\widetilde \eta.$ Thus $\widetilde\eta$ is a universal transgression form
for the delooping sequence, in that it restricts to the Chern character on 
$G^{-\infty}_{\sus}(\phi;E)$ and $d\widetilde\eta$ is basic; it is the
pull-back of the odd Chern character $\Ch_{\odd}$ on $G^{-\infty}(\phi;E).$

In \S\ref{deLoop} the smooth delooping sequence for K-theory is described.
The universal Chern forms on the odd and even classifying spaces are
constructed in \S\ref{U-Chern}; the regularization to a universal eta form
on the half-open loop group is carried out in \S\ref{U-eta}. The
constructions of Chern forms is extended to the classifying bundle given by
a fibration in \S\ref{G-Chern}. The bundle of invertible perturbations for
a self-adjoint elliptic family, or more generally an elliptic family of
product-type suspended operators, is introduced in \S\ref{Odd-ell} and in
\S\ref{EtaFam} the eta forms are generalized to this case and further
extended in \S\ref{Exteta}. The index formula, Theorem~\ref{efatopi.4}, is
proved in \S\ref{Index-form}. The realization of the exponentiated eta
invariant, the $\tau$-invariant, as a determinant is discussed in
\S\ref{Dets} and the adiabatic determinant of a doubly-suspended family is
discussed in \S\ref{DoubDet}. This is used to construct a smooth and
primitive form of the determinant line bundle over the even classifying
space in \S\ref{Det-line}. The K-theory gerbe is realized as a bundle gerbe
in the sense of Murray \cite{Murray1} in \S\ref{K-gerbe} and the geometric
version of this gerbe for an elliptic family is described in
\S\ref{Ell-gerbe}.  Finally, in \S\ref{bs.0}, we discuss the relationship between the eta forms as introduced here and the eta forms of Bismut-Cheeger \cite{Bismut-Cheeger}.

\paperbody
\section{Delooping sequence}\label{deLoop}

We first consider the `universal' case with constructions directly over
classifying spaces. Despite the infinite-dimensional base, this setting is
a little simpler than the geometric case of a fibration since there is no
twisting by diffeomorphisms. Let $Z$ be a compact manifold with $\dim Z>0$
and let $E\longrightarrow Z$ be a complex vector bundle over it. The
algebra of smoothing operators on sections of $E$ is
\begin{equation*}
\Psi^{-\infty}(Z;E)=\CI(Z^2;\Hom(E)\otimes\Omega_R)
\label{efatoi.16}\end{equation*}
where $\Omega_{R}=\pi^{*}_{R}\Omega$ is the pull-back of the density bundle by
the projection $\pi_{R}: Z\times Z\longrightarrow Z$ onto the right factor and 
$\Hom(E)= \pi_{R}^{*}E'\otimes \pi_{L}^{*}E$ with
$\pi_{L}:Z\times Z\longrightarrow Z$ the projection onto the left factor.
The  product is given by the integral 
\begin{equation}
(A\circ B)(z,z')=\int_{Z}A(z,z'')B(z'',z').
\label{efatopi.89}\end{equation}

The topological group
\begin{equation}
G^{-\infty}(Z;E)=
\big\{A\in \Psi^{-\infty}(Z;E);\ \exists\ (\Id+A)^{-1}=\Id+B,\
B\in \Psi^{-\infty}(Z;E)\big\}
\label{efatoi.17}\end{equation}
is an open dense subset and is classifying for odd K-theory. The `suspended' (or
flat-pointed loop) group 
\begin{equation}
G^{-\infty}_{\sus}(Z;E)=
\big\{A\in\cS(\bbR_\tau\times Z^2;\Hom(E)\otimes\Omega _R);
A(\tau)\in G^{-\infty}(Z;E)\big\}
\label{efatoi.18}\end{equation}
is therefore classifying for even K-theory. It is an open (and dense)
subspace of the Schwartz functions on $\bbR$ with values in
$\Psi^{-\infty}(Z;E).$

Thus, for any other manifold $X,$ the sets of equivalence classes of
(smooth) maps reducing to the identity outside a compact set under (smooth)
homotopy through such maps are the K-groups:
\begin{equation*}
\begin{gathered}
\Kco(X)=\big\{f\in\CI(X;G^{-\infty}(Z;E));f=\Id\Mon X\setminus K,\ K\Subset X
\big\}/\sim,
\\
\Kce(X)=\big\{f\in\CI(X;G^{-\infty}_{\sus}(Z;E));f=\Id\Mon X\setminus K,\
K\Subset X\big\}/\sim.
\end{gathered}
\label{efatoi.15}\end{equation*}

By definition, Schwartz functions are `flat at infinity' and we
introduce a larger space of functions which are Schwartz at $-\infty$ but
more generally `flat to a constant' at $+\infty$ and the corresponding group
\begin{multline}
\tilde G^{-\infty}_{\sus}(Z;E)=\{A\in\CI(\bbR_\tau\times
Z^2;\Hom(E)\otimes\Omega _R);\lim_{\tau\to-\infty}A(\tau)=0,\\
\frac{d A(\tau)}{d\tau}\in\cS(\bbR_\tau\times Z^2;\Hom(E)\otimes\Omega _R),\
A(\tau)\in G^{-\infty}(Z;E)\ \forall\ \tau\in[-\infty,\infty]\}.
\label{efatopi.7}\end{multline}
Thus $A$ can be recovered from its derivative,
\begin{equation}
A(\tau)=\int_{-\infty}^\tau \frac{dA(s)}{ds}ds.
\label{efatopi.31}\end{equation}
Moreover, there is a well-defined map `restriction to $\tau=\infty$', 
\begin{equation}
R_{\infty}:\tilde G^{-\infty}_{\sus}(Z;E)\longrightarrow G^{-\infty}(Z;E)
\label{efatopi.32}\end{equation}
which is surjective since $G^{-\infty}(Z;E)$ is connected and a general
curve between two points can be smoothed and flattened at the ends.

The delooping sequence in the present context is the short exact sequence
of groups
\begin{equation}
\xymatrix{
G^{-\infty}_{\sus}(Z;E)\ar[r]^{\iota}&
\tilde G^{-\infty}_{\sus}(Z;E)\ar[r]^{R_{\infty}}&
G^{-\infty}(Z;E)}
\label{efatopi.8}\end{equation}
where the map to the quotient group is explicitly given by
\eqref{efatopi.32} and the flatness of the paths at $+\infty$ ensures
exactness in the middle.

\begin{lemma}\label{contractibility.1}
The group $ \tilde G^{-\infty}_{\sus}(Z;E)$ is contractible.
\end{lemma}

\begin{proof} It is only the `flatness at infinity' of the elements of
  $\tilde G^{-\infty}_{\sus}(Z;E)$ that distinguishes this result from the
  standard contractibility, by shortening the curve, of the pointed path
  space of a group. To maintain this condition during the contraction,
  first identify $(-\infty,\infty)$ by radial compactification with the
  interior of $[0,1].$ Since the singularities in the compactification are
  swamped by the rapid vanishing of the derivatives at the end points, this
  gives the alternative description of the group as
\begin{equation}
\begin{gathered}
\tilde G^{-\infty}_{\sus}(Z;E)=\\
\{ a\in \CI([0,1]_x;G^{-\infty}(Z;E));
\frac{da}{dx} \in \dot{\mathcal{C}}^{\infty}([0,1]; \Psi^{-\infty}(Z;E)), 
\ a(0)=0 \},
\end{gathered}
\label{contr.2}\end{equation}
where  $\dot{\mathcal{C}}^{\infty}([0,1]; \Psi^{-\infty}(Z;E))$ is the space
of smooth functions vanishing together with all their derivatives at $x=0$
and $x=1.$

Now, let $\rho:[0,1]\longrightarrow [0,1]$ be a smooth function with
$\rho(0)=0$ and $\rho(x)=1$ near $x=1$ and consider the homotopy
\begin{equation}
\psi_t(x)=
\begin{cases}
2t\rho(x), & t\in [0,\frac{1}{2}], \\
\rho(x)+ (2t-1)(x-\rho(x)), & t\in [\frac{1}{2},1],
\end{cases}
\label{contr.3}\end{equation}
between the constant map, $\psi_{0}(x),$ and the identity map $\psi_{1}(x)=x.$
Note that $\psi_t(1)=1$ for $\frac{1}{2}\le t\le 1$ and $\psi_t$ is flat at $1$
for $0\le t\le \frac{1}{2}.$ It follows that the composite $f(\psi_t(x))$ with
$f\in\CI([0,1])$ is flat at $x=1$ for all $t$ if $f$ is flat at $x=1.$ Thus
composition
\begin{equation}
\Psi_{t}:\tilde G^{-\infty}_{\sus}(Z;E)\ni a\longmapsto a\circ
\psi_{t}\in\tilde G^{-\infty}_{\sus}(Z;E)
\label{contr.4}\end{equation}
gives the desired deformation retraction to the identity element.
\end{proof}

This argument is not limited to this particular group and holds in greater
generality.

\section{Universal Chern forms}\label{U-Chern}

The group $G^{-\infty}(Z;E),$ identified as an open dense subset of
$\Psi^{-\infty}(Z;E)$, is an infinite dimensional manifold modelled on the
Fr\'echet space $\CI(Z^2;\Hom(E)\otimes\Omega _R).$ We shall fix the space
of smooth functions on $G^{-\infty}(Z;E)$ and more generally the smooth
sections of form bundles and other tensor bundles.

First, it is natural to identify the tangent space at any point with the
linear space in which the group is embedded. Then the cotangent space can
be identified with its dual, $\CmI(Z^2;\Hom(E')\otimes\Omega_L),$ the space
of distributional sections, where $\Omega_L$ is the left density bundle. Thus
\begin{equation}
T^*_aG^{-\infty}(Z;E)=\CmI(Z^2;\Hom(E')\otimes\Omega_L)
\label{efatopi.10}\end{equation}
with the duality between smooth tangent and cotangent fibres given by
distributional pairing. This can be written formally as bundle pairing
followed by integration
\begin{equation}
\begin{gathered}
T^*_aG^{-\infty}(Z;E)\times T_aG^{-\infty}(Z;E)\ni(\alpha,B)\longrightarrow
\alpha \cdot B\in\bbC,\\
\alpha \cdot B=\int_{Z^2}\alpha(z,z')B(z,z').
\end{gathered}
\label{efatopi.11}\end{equation}

Having defined the tangent and cotangent fibres at each point,
the fibres of the cotensor bundles are interpreted as completed tensor
products. Thus
\begin{equation}
(T^*)^{\otimes k}_a=\CmI(Z^{2k};\bigotimes_j\pi^*_j\Hom(E')\otimes\Omega _{kL})
\label{efatopi.13}\end{equation}
where $\Omega _{kL}$ is the tensor product of the (trivial) real line
bundles on each left factor of all the pairs and the homomorphism bundle is
lifted from each pair of factors.

Since $G^{-\infty}(Z;E)$ is a metric space with the topology
induced from $\Psi^{-\infty}(Z;E),$ continuity for functions is immediately
defined.  More generally, continuity for sections of any of the tensor
bundles is defined by insisting that a $k$-cotensor field should be a
continuous map from the metric space $G^{-\infty}(Z;E)$ (or indeed any
subset of it) into the distributional space \eqref{efatopi.13} in the
strong sense that it should map locally into some fixed Sobolev, hence
Hilbert, space
$$
H^{-N}(Z^{2k};\bigotimes_j\pi^*_j\Hom(E')\otimes\Omega
_{kL})
$$
and should be continuous for the metric topologies. The meaning of
directional derivatives is then clear. For a map to be $C^1$, we insist that
all directional derivatives exist at each point, that they are jointly
defined by an element of the next higher tensor space, i.e.\ distribution
in two more variables, and that the resulting section of this tensor bundle
is also continuous. Then infinite differentiability is defined by 
iteration.

The form bundles are defined, as usual, as the totally antisymmetric parts
of the corresponding cotensor bundles. Smoothness as a form is smoothness
as a cotensor field. The deRham differential is the map from smooth
$k$-forms to smooth $(k+1)$-forms given in the usual way by differentiation
followed by antisymmetrization.

If $F:G^{-\infty}(Z;E)\longrightarrow \bbC$ is smooth and $b\in
G^{-\infty}(Z;E)$ then $L(b)^*F(a)=F(ba)$ is also smooth, as is $R(b)^*F$
defined by $R(b)^*F(a)=F(ab^{-1}).$ Thus $G^{-\infty}(Z;E)$ acts on its
space of smooth functions, as in the setting of finite dimensional Lie
groups. These actions extend to cotensor fields and hence to forms.

The universal odd Chern character is given by a slight reinterpretation of
the standard finite-dimensional formula
\begin{multline}
\Ch_{\odd}(a)=\\
\frac{1}{2\pi i}\Tr
\left(\int_0^1a^{-1}da\exp{\left(\frac{t(1-t)(a^{-1}da)^2}{2\pi i}\right)}dt\right)
\in\CI(G^{-\infty}(Z;E);\Lambda^{\odd}).
\label{efatoi.20}\end{multline}
Namely expanding out the exponential in formal power series and carrying
out the resulting integrals reduces this to an infinite sum 
\begin{equation}
\Ch_{\odd}(a)=\sum\limits_{k=0}^\infty c_k\Tr((a^{-1}da)^{2k+1}),\ c_k= \frac{1}{(2\pi i)^{k+1}}
  \frac{k!}{(2k+1)!}.
\label{efatopi.24}\end{equation}
Here, each $da$ is the identification of the tangent space at $a$ with
$\Psi^{-\infty}(Z;E)$ -- so can be thought of as the differential of the
identity. Thus, for any $2k+1$ elements $b_{j}\in\Psi^{-\infty}(Z;E),$ the
evaluation on $(T_a)^{\otimes(2k+1)}$ of an individual term is
\begin{multline}
\Tr((a^{-1}da)^{2k+1})(b_1,\dots,b_{2k+1})=\\
\sum\limits_{\sigma} (-1)^{\sign(\sigma)}
\Tr(a^{-1}b_{\sigma(1)}a^{-1}b_{\sigma(2)}\dots a^{-1}b_{\sigma (2k+1)}).
\label{efatopi.25}\end{multline}
The trace is well defined since the product is an element of
$\Psi^{-\infty}(Z;E).$ It is also defined at each point by a distribution, as
required above, and the same is true of all derivatives. Namely at
each point the distribution defining this form is just the total
antisymmetrization (of variables in pairs) of
\begin{equation}
A(z'_{2k+1},z_1)A(z'_{1},z_2)A(z'_{2},z_3)\dots A(z'_{2k},z_{2k+1})
\label{efatopi.26}\end{equation}
where $A$ is the Schwartz kernel of $a^{-1}.$ Note that while this
\emph{is} smooth in the sense described above, the kernel representing the
form at a given point is not a smooth function because of the presence of
the identity factors in the operators. Due to the identity
\begin{equation}
\frac{d}{dt} a_t^{-1}=-a^{-1}\frac{da_t}{dt} a^{-1}
\label{efatopi.27}\end{equation}
differentiation gives a similar form, but with less symmetrization, 
with respect to parameters. Thus \eqref{efatopi.25} defines a form in each
odd degree.

As a result of antisymmetrization the form corresponding to
\eqref{efatopi.25} for an even power is identically zero. Moreover the
computation of the deRham differential, based on the identities
\eqref{efatopi.27}, $d^2a=0$ and $d(a^{-1}da a^{-1})=0$ yields
\begin{equation}
d\Tr\left((a^{-1}da)^{2k+1}\right)=
-\Tr\left((a^{-1}da)^{2k+2}\right)=0\Longrightarrow
d\Ch_{\odd}=0 
\label{efatopi.28}\end{equation}
globally on $G^{-\infty}(Z;E).$ The Chern character \eqref{efatoi.20} is
universal in the sense that if $f:X\longrightarrow G^{-\infty}(Z;E)$ is any
smooth map from a compact manifold $X$, then 
\begin{equation}
      [f^{*}\Ch_{\odd}]=\Ch_{\odd}([f])\in H^{\odd}(X;\bbC)
\label{efatoi.28b}\end{equation}
represents the odd Chern character of the $K$-class defined by the homotopy
class $[f]$ of $f.$

The abelian group structure on $\Ko(X)$ is derived from the non-abelian
group structure on $G^{-\infty}(Z;E)$ and in particular the linearity of
the odd Chern character is a consequence of the following result. Here we
say that a form on a product of two (infinite-dimensional) manifolds
$M_1\times M_2$ `has no pure terms' if it vanishes when restricted to
$\{p_1\}\times M_2$ or $M_1\times\{p_2\}$ for any points $p_1\in M_1$ or
$p_2\in M_2.$ 

\begin{proposition}\label{efatopi.140} There is a smooth form
  $\delta_{\even}$ on $G^{-\infty}(Z;E)\times G^{-\infty}(Z;E)$ of even
  degree which has no pure terms, vanishes when pulled back to the `product
  diagonal' $\{(a,a^{-1})\}$ and is such that in terms of pull-back under
  the product map and two projections:
\begin{equation}
\xymatrix{
&G^{-\infty}(Z;E)\\
&G^{-\infty}(Z;E)\times G^{-\infty}(Z;E)\ar[u]^m\ar[dr]_{\pi_R}\ar[dl]^{\pi_L}\\
G^{-\infty}(Z;E)&&G^{-\infty}(Z;E)
}
\label{efatopi.141}\end{equation}
the form in \eqref{efatoi.20} satisfies
\begin{equation}
m^*\Ch_{\odd}=\pi_L^*\Ch_{\odd}+\pi_R^*\Ch_{\odd}+d\delta_{\even}.
\label{efatopi.142}\end{equation}
\end{proposition}

\begin{proof} For any two bundles the group $G^{-\infty}(Z;E)\oplus
  G^{-\infty}(Z;F)$ can be identified as the diagonal subgroup of
  $G^{-\infty}(Z;E\oplus F)$ and the Chern form restricted to this subgroup
  clearly splits as the direct sum. So, to prove \eqref{efatopi.142} we
  work on $E\oplus E$ and take a homotopy which connects $ab$
  acting on the left factor of $E,$ so as $ab\oplus\Id$ on $E\oplus E,$ with
  $a\oplus b$ acting on $E\oplus E.$ This can be constructed in terms of a
  rotation between the two factors. Thus 
\begin{equation}
M(t)=\begin{pmatrix}\cos t&\sin t\\
-\sin t&\cos t
\end{pmatrix},\ t\in[0,\pi/2]
\label{efatopi.156}\end{equation}
is such that 
\begin{equation}
B(t)=M^{-1}(t)\begin{pmatrix}b&0\\0&\Id
\end{pmatrix}M(t)\Msatisfies B(0)=\begin{pmatrix}b&0\\0&\Id
\end{pmatrix},\ B(\pi/2)=\begin{pmatrix}\Id&0\\0&b
\end{pmatrix}.
\label{efatopi.157}\end{equation}

Using this family, consider the map 
\begin{equation}
H:[0,1]\times G^{-\infty}(Z;E)\times G^{-\infty}(Z;E)\longmapsto A(0)B(t)\in
G^{-\infty}(Z;E\oplus E).
\label{efatopi.158}\end{equation}
It follows that the form $\alpha=H^*\Ch_{\odd}$ is a closed form on the product
and hence decomposing in terms of the factor $[0,1],$  
\begin{equation}
\alpha =dt\wedge \alpha _1(t)+\alpha _2(t)
\label{efatopi.159}\end{equation}
where the $\alpha _i$ are smooth 1-parameter families of forms
on $G^{-\infty}(Z;E)\times G^{-\infty}(Z;E),$  
\begin{equation}
d\alpha_2=0,\ d\alpha _1=\frac{\pa}{\pa t}\alpha _2
\label{efatopi.160}\end{equation}
where $d$ is now the deRham differential on $G^{-\infty}(Z;E)\times
G^{-\infty}(Z;E).$ Thus, setting
\begin{equation}
\delta_{\even} =-\int_0^{\pi/2}\alpha _1(t)dt,
\label{efatopi.161}\end{equation}
\eqref{efatopi.142} follows.

Now, if $a$ is held constant, $H^*\Ch_{\odd}$ is independent of $a$ and
reduces to the Chern character for $B(t).$ It follows that the individual
terms in $\alpha _1$ are multiples of
\begin{multline}
\Tr\bigg(\begin{pmatrix}\Id&0\\0&0
\end{pmatrix}
\big(\frac{\pa M(t)}{\pa t} M^{-1}(t)((db) b^{-1})^{2k}\big)\\
-\begin{pmatrix}\Id&0\\0&0
\end{pmatrix}
\left(M^{-1}(t)\frac{\pa M(t)}{\pa t} (b^{-1}db)^{2k}\right)\bigg).
\label{efatopi.162}\end{multline}
Since $\frac{\pa M(t)}{\pa t} M^{-1}(t)$ and $M^{-1}(t)\frac{\pa M(t)}{\pa t}$
 are off-diagonal this vanishes. A similar argument applies if $b$ is held
 constant, so $\delta_{\even}$ in \eqref{efatopi.161} is without pure terms. 

Under inversion, $a\longmapsto a^{-1},$ $\Ch_{\odd}$ simply changes sign, so
under the involution $I:(a,b)\longmapsto (b^{-1},a^{-1})$ both the left
side and the two Chern terms, together, on the right change sign. Thus
$\delta=\delta _{\even}$ can be replaced by its odd part under this
involution, $\ha (\delta -I^*\delta),$ which ensures that it vanishes when
pulled back to the submanifold left invariant by $I,$ namely
$\{b=a^{-1}\}.$ It still is without pure terms so the proposition is
proved.
\end{proof}

The discussion of the suspended group $G^{-\infty}_{\sus}(Z;E)$ is
similar. Namely the tangent space is the space of Schwartz
sections $\cS(\bbR\times Z^2;\Hom(E)\otimes\Omega _R)$ which can be
identified, by radial compactification of the line, with $\dCI([-1,1]\times
Z^2;\Hom(E)\otimes\Omega _R)\subset\CI([-1,1]\times
Z^2;\Hom(E)\otimes\Omega _R),$ consisting of the space of smooth sections
on this manifold with boundary, vanishing to infinite order at both
boundaries. The dual space is then the space of Schwartz distributions
$\cSp(\bbR\times Z^2;\Hom(E')\otimes\Omega _L),$ or in the compactified
picture the space of extendible distributional sections. Apart from these
minor alterations, the discussion proceeds as before and the even Chern
forms, defined by pull-back and integration are
\begin{equation}
\Ch_{\even}= p_{*}(\ev^*\Ch_{\odd})\in
\CI(G^{-\infty}_{\sus}(Z;E);\Lambda^{\even}),
\label{efatopi.23}\end{equation}
where 
\begin{equation}
\ev:\bbR\times G^{-\infty}_{\sus}(Z;E)\ni(s,A)\longmapsto A(s)\in G^{-\infty}(Z;E)
\label{efatopi.9}\end{equation}
is the evaluation map and $p_{*}$ is the pushforward map along the fibres of the
projection $p:\bbR\times G^{-\infty}_{\sus}(Z;E)\longrightarrow
G^{-\infty}_{\sus}(Z;E)$ on the right factor. So, at least formally, 
\begin{equation}
\Ch_{\even}=\frac1{2\pi i}\int_{\bbR}\int_0^1
\Tr
\left((a^{-1}da)
\exp{\left(\frac{t(1-t)(a^{-1}da)^2}{2\pi i}\right)}\right)dt,
\label{efatopi.65}\end{equation}
where the outer integral mean integration with respect to $\tau$ of the
coefficient of $d\tau.$

The analogue of Proposition~\ref{efatopi.140} for the even Chern character
follows from that result. Namely if we consider the corresponding product
map, pointwise in the parameter, and projections: 
\begin{equation}
\xymatrix{
&G^{-\infty}_{\sus}(Z;E)\\
&G^{-\infty}_{\sus}(Z;E)\times
G^{-\infty}_{\sus}(Z;E)\ar[u]^m\ar[dr]_{\pi_R}\ar[dl]^{\pi_L}\\
G^{-\infty}_{\sus}(Z;E)&&G^{-\infty}_{\sus}(Z;E)
}
\label{efatopi.143}\end{equation}
then there is a smooth form $\delta _{\odd}$ on the product group such that
\begin{equation}
m^*\Ch_{\even}=\pi_L^*\Ch_{\even}+\pi_R^*\Ch_{\even}+d\delta _{\odd}.
\label{efatopi.144}\end{equation}
This odd form can be constructed from $\delta _{\even}$ using the pull-back
and push-forward operations for the (product) evaluation map 
\begin{multline}
\Ev:\bbR\times G^{-\infty}_{\sus}(Z;E)\times G^{-\infty}_{\sus}(Z;E)\ni(\tau,a,b)
\longrightarrow\\
(a(\tau),b(\tau))\in G^{-\infty}(Z;E)\times G^{-\infty}(Z;E)
\label{efatopi.145}\end{multline}
as 
\begin{equation}
\delta _{\odd}= -\int_{\bbR} \delta '(\tau)d\tau,\ \Ev^*\delta_{\even}=d\tau\wedge\delta
'(\tau)+\delta ''(\tau). 
\label{efatopi.146}\end{equation}
Since the forms are Schwartz in the evaluation parameter, the additional
term  
\begin{equation}
\int_{\bbR}\frac{\pa}{\pa \tau}\delta ''(\tau)=0
\label{efatopi.147}\end{equation}
and \eqref{efatopi.144} follows; note that it does not follow from the fact
that $\delta_{\even}$ has no pure terms that this is true of $\delta
_{\odd}$ -- and it is not!

Smoothness of forms in the sense discussed above certainly implies that the
pull-back of such a form to a finite dimensional manifold, under a smooth
map $Y\longrightarrow G^{-\infty}(Z;E)$ is smooth on $Y$ and closed if the
form on $G^{-\infty}(Z;E)$ is closed. Thus if $f:Y\longrightarrow
G^{-\infty}(Z;E)$ is a representative of $[f]\in \Ko(Y)$ then
$f^*\Ch_{\odd}$ is a sum of closed odd-degree forms on $Y.$ The cohomology
class is constant under homotopy of the map. Indeed, a homotopy between
$f_0$ and $f_1$ is a map $F:[0,1]_r\times Y\longrightarrow
G^{-\infty}(Z;E).$ The fact that the Chern form pulls back to be closed
shows that $F^*\Ch_{\odd}$ is of the form
\begin{equation}
\begin{gathered}
A(r)+dr\wedge B(r),\ d_YA(r)=0,\ \frac{dA(r)}{dr}=d_YB(r)\\
\Longrightarrow
A(1)-A(0)=d\int_0^1B(r)dr.
\end{gathered}
\label{efatopi.30}\end{equation}
Thus cohomology classes in the even case are also homotopy invariant and
these universal Chern forms define a map from K-theory to cohomology. This
is the Chern character. The theorem of Atiyah and Hirzebruch
shows that the combined even and odd Chern characters give a multiplicative
isomorphism
\begin{equation*}
\Ch:\big(\Ke(X)\oplus\Ko(X)\big)\otimes\bbC\longrightarrow\sH(X;\bbC).
\label{efatoi.19}\end{equation*}

\section{Universal eta form}\label{U-eta}

As a link between the odd and even universal Chern characters defined above
on the end groups in \eqref{efatopi.8}, we consider the corresponding eta
form on $\tilde G^{-\infty}_{\sus}(Z;E).$ It has the same formal definition
as the even Chern character but now lifted to the larger (and contractible)
group. This consists of paths in $G^{-\infty}(Z;E)$ so there is still an
evaluation map
\begin{equation}
\tEv:\bbR\times \tilde
G^{-\infty}_{\sus}(Z;E)\ni(s,A)\longmapsto A(s)\in G^{-\infty}(Z;E)
\label{efatopi.49}\end{equation}
just as in \eqref{efatopi.9}.

\begin{definition}\label{efatopi.19} The universal eta form on $\tilde
  G_{\sus}^{-\infty}(Z;E)$ is defined as in \eqref{efatopi.23} but interpreted
  on the group $\widetilde G^{-\infty}_{\sus}(Z;E)$ with the evaluation map
  \eqref{efatopi.49}
\begin{equation}
\tilde\eta= \tilde p_{*}\left(\tEv^*\Ch_{\odd}\right)
 \in\CI(\widetilde G^{-\infty}_{\sus}(Z;E);\Lambda^{\even})
\label{efatopi.20}\end{equation}
and with $\tilde p_{*}$ the push-forward map corresponding to integration
along the fibres of the projection $\tilde p: \bbR\times \widetilde
G^{-\infty}_{\sus}(Z;E)\longrightarrow \widetilde G^{-\infty}_{\sus}(Z;E).$
\end{definition}

Integration on the fibres here is well defined since, in the integrand --
which is the contraction with $\pa/\pa\tau$ -- necessarily one of the terms
is differentiated with respect to the suspension parameter, which has
the effect of removing the constant term at infinity. Thus the integral in
\eqref{efatopi.20} still converges rapidly.

If $X$ is a compact smooth manifold and if $a: X\longrightarrow
\widetilde{G}^{-\infty}(Z;E)$ is a smooth map, the associated eta form is
\begin{equation}
   \eta(a)= a^{*}\tilde\eta.
\label{etaform}\end{equation}

Now, consider the diagram analogous to \eqref{efatopi.143} but for the
extended group, and hence with an additional map corresponding to
restriction to $t=\infty$ in each factor:

\begin{equation}
\xymatrix{
&\tilde G^{-\infty}_{\sus}(Z;E)\\
&\tilde G^{-\infty}_{\sus}(Z;E)\times
\tilde G^{-\infty}_{\sus}(Z;E)
\ar[u]^m\ar[dr]_{\pi_R}\ar[dl]^{\pi_L}\ar[dd]^{R_{\infty}\times
R_{\infty}}\\
\tilde G^{-\infty}_{\sus}(Z;E)&
&\tilde G^{-\infty}_{\sus}(Z;E)\\
&G^{-\infty}(Z;E)\times G^{-\infty}(Z;E).
}
\label{efatopi.148}\end{equation}

\begin{proposition}\label{efatopi.21} The eta form in \eqref{efatopi.20}
restricts to $\Ch_{\even}$ on $G^{-\infty}_{\sus}(Z;E),$ satisfies the
identity  
\begin{equation}
m^*\tilde\eta=\pi_L^*\tilde\eta+\pi_R^*\tilde\eta+d(\tilde\delta _{\odd})
+(R_{\infty}\times R_{\infty})^*\delta _{\even}
\label{efatopi.149}\end{equation}
where $\tilde\delta_{\odd}$ is a smooth form on $\tilde
G^{-\infty}_{\sus}(Z;E)\times\tilde G^{-\infty}_{\sus}(Z;E)$ which
restricts to $\delta_{\odd}$ on $G^{-\infty}_{\sus}(Z;E)\times
G^{-\infty}_{\sus}(Z;E)$ and moreover $\tilde\eta$ has basic differential
\begin{equation}
d\tilde\eta =R_{\infty}^*\Ch_{\odd}
\label{efatopi.22}\end{equation}
where $R_{\infty}$ is the quotient map in \eqref{efatopi.8}.
\end{proposition}

\begin{proof} To compute the differential of the eta form, write the
  pull-back under $\tEv$ as in \eqref{efatopi.30}: 
\begin{equation}
\tEv^*\Ch_{\odd}=A(\tau)+d\tau\wedge B(\tau) \Longrightarrow
dB(\tau)=\frac{dA(\tau)}{d\tau}. 
\label{efatopi.33}\end{equation}
Since
\begin{equation*}
\tilde\eta=\int_{\bbR}B(\tau),\ d\tilde\eta=\int_{\bbR}dB(\tau)d\tau=
\int_{\bbR}\frac{dA(\tau)}{d\tau}d\tau=A(\infty)=R_{\infty}^*\Ch_{\odd}.
\label{efatopi.34}\end{equation*}
This proves \eqref{efatopi.22}.

Similarly, as in the proof of \eqref{efatopi.144}, pulling back the
corresponding additivity formula, \eqref{efatopi.142}, for the odd Chern
character gives \eqref{efatopi.149} with the additional term arising from
the integral which vanishes as in \eqref{efatopi.147} on the suspended subgroup.
\end{proof}

\section{Geometric Chern forms}\label{G-Chern}

Next we pass to a discussion of the `geometric case'. Fix a connection on
the fibration \eqref{efatopi.1}. That is, choose a smooth splitting 
\begin{equation}
TM= T^HM\oplus T(M/Y)
\label{connection.1}\end{equation}
where the subbundle $T^{H}M$ is necessarily isomorphic to $\phi^{*}TY.$
Also choose a connection $\nabla^{E}$ on the complex vector bundle
$E\longrightarrow M.$ Consider the infinite-dimensional bundle
\begin{equation}
\CI(\phi;E)\longrightarrow Y
\label{gcf.1}\end{equation}
which has fibre $\CI(Z_y;E_y),$ $Z_y=\phi^{-1}(y),$ $E_y=E\big|_{Z_y}$ at $y\in
Y$ and space of smooth global sections written $\CI(M/Y;E),$ which is
canonically identified with $\CI(M;E).$ The choice of connections induces a
connection on $\CI(\phi;E)$ through the covariant differential
\begin{equation}
\nabla^{\phi,E}_{X}u= \nabla^{E}_{X_{H}}\tilde u,\ \CI(M/Y;E)\ni u=\tilde
u\in\CI(M;E),
\label{gcf.2}\end{equation}
where $X_{H}$ is the horizontal lift of $X\in\CI(Y;TY).$ The curvature of
this connection is a 2-form on the base with values in the first-order
differential operators on sections of $E$ on the fibres
\begin{equation}
\omega=(\nabla^{\phi,E})^{2}\in\Lambda^2Y\otimes_{\CI(Y)}\Diff1(M/Y;E).
\label{gcf.3}\end{equation}

This covariant differential can be extended to the bundle $\Psi^m(\phi;E),$
for each $m$ including $m=-\infty,$ which has fibre $\Psi^{m}(Z_y, E_y)$ at
$y,$ and space of global smooth sections $\Psi^{m}(M/Y;E)$ through its
action on $\CI(M;E):$
\begin{equation}
\nabla^{\phi,E}Q= [\nabla^{\phi,E},Q],\ Q\in\Psi^{m}(M/Y;E).
\label{gcf.4}\end{equation}
The curvature of the induced connection is given by the commutator action
of the curvature
\begin{equation}
(\nabla^{\phi,E})^{2}= [\omega,\cdot].
\label{gcf.5a}\end{equation}

Let $\pi: G^{-\infty}(\phi;E)\longrightarrow Y$ be the infinite-dimensional
bundle over $Y$ with fibre
\begin{equation}
G^{-\infty}(Z_y;E_y)=\big\{\Id_{E_y}+Q;Q\in \Psi^{-\infty}(Z_y;E_y),\
\Id_{E_y}+Q\text{ is invertible} \big\}.
\label{gcf.5}\end{equation}
This is naturally identified with an open subbundle of
$\Psi^{-\infty}(\phi;E)\subset\Psi^{m}(\phi;E)$ and as such has an induced
covariant differential. If $\sigma\in G^{-\infty}(M/Y;E)$ is a global
section, the corresponding odd Chern character is
\begin{equation}
\begin{gathered}
\Ch_{\odd}(\sigma,\nabla^{\phi,E})= \frac{1}{2\pi i} \Tr \left( \int_{0}^{1} 
(\sigma^{-1}\nabla^{\phi,E}\sigma)\exp\left(\frac{w(s,\sigma,\nabla^{\phi,E})}
{2\pi i}\right) ds
\right),\Mwhere \\
w(s,\sigma,\nabla)=
s(1-s) (\sigma^{-1}\nabla\sigma)(\sigma^{-1}\nabla\sigma) + (s-1) \omega
-s \sigma^{-1}\omega \sigma.          
\end{gathered} 
\label{gcf.7}\end{equation}
Even though the curvature $\omega$ from \eqref{gcf.3} is not of trace
class, the term  $\sigma^{-1}\nabla^{\phi,E}\sigma$ is a $1$-form with values in
smoothing operators, the identity being annihilated by the covariant
differential, so the argument of $\Tr$ is a smoothing operator.

The form in \eqref{gcf.7} is the pull-back under the section $\sigma$ of a
`universal' odd Chern character on the total space of the bundle. To see
this, first pull the bundle back to its own total space
\begin{equation}
   \pi^{*}G^{-\infty}(\phi;E)\longrightarrow G^{-\infty}(\phi;E).
\label{gcf.8}\end{equation}
This has a tautological section
\begin{equation}
 a:  G^{-\infty}(\phi;E)\longrightarrow \pi^{*}G^{-\infty}(\phi;E)
\label{gcf.9}\end{equation}
and carries the pulled back covariant differential $\tilde\nabla^{\phi,E}=
\pi^{*}\nabla^{\phi,E}.$ The geometric odd Chern character on
$G^{-\infty}(\phi;E)$ is
\begin{equation}
\begin{gathered}
\Ch_{\odd}(\tilde\nabla^{\phi,E})=
\frac{1}{2\pi i} \Tr \left( \int_{0}^{1} 
a^{-1}\tilde\nabla^{\phi,E} a
\exp\left({ \frac{w(s,a,\tilde\nabla^{\phi,E})}{2\pi i}}\right)ds
\right),\Mwhere \\
w(s,a,\tilde\nabla^{\phi,E})=
s(1-s) (a^{-1}\tilde\nabla^{\phi,E} a)(a^{-1}\tilde\nabla^{\phi,E} a) +
(s-1) \tilde\omega
      -s a^{-1}\tilde\omega a;
\end{gathered}
\label{gcf.10}\end{equation}
here $\tilde\omega= \pi^{*}\omega$ is the pull-back of the curvature. This
clearly has the desired universal property for smooth sections:
\begin{equation}
\Ch_{\odd}(\sigma,\nabla^{\phi,E})= \sigma^{*}\Ch_{\odd}(\tilde\nabla^{\phi,E}).
\label{gcf.11}\end{equation}

The basic properties of the geometric Chern character are well known and
discussed, for example, in \cite{rccbif}. In particular of course, the
forms are closed. This follows from identities for the forms
$w=w(s,a,\tilde\nabla^{\phi,E})$ and $\theta =a^{-1}\tilde\nabla^{\phi,E}
a$ in \eqref{gcf.10} which will be used below. Namely the Bianchi identity
for the connection implies (\cf (3.5) in \cite{rccbif}) that 
\begin{equation}
\begin{gathered}
\tilde\nabla^{\phi,E} w=s[w,\theta ]\text{ and hence}\\
\begin{aligned}
\tilde\nabla^{\phi,E} &\left(\theta \exp(\frac{w}{2\pi i})\right)=
- \frac{dw}{ds}\exp\left(\frac{w}{2\pi i})\right) -
s[\theta \exp(\frac{w}{2\pi i}),\theta], \\
&=-2\pi i\frac{d}{ds}\exp(\frac{w}{2\pi i})+\int_0^1[e^{\frac{(1-r)w}{2\pi i}}
    ,\frac{dw}{ds}e^{\frac{rw}{2\pi i}}]dr-s[\theta \exp(\frac{w}{2\pi i}),\theta].
\end{aligned}
\end{gathered}
\label{efatopi.105}\end{equation}
All the commutators have vanishing trace so 
\begin{multline}
d\Ch_{\odd}=\frac{1}{2\pi i} \Tr \tilde\nabla^{\phi,E}\left( \int_{0}^{1} 
a^{-1}\tilde\nabla^{\phi,E} a
\exp\left({ \frac{w(s,a,\tilde\nabla^{\phi,E})}{2\pi i}}\right)ds
\right)\\
=
-\Tr\int_{0}^{1} \frac{d}{ds}\exp\left({
  \frac{w(s,a,\tilde\nabla^{\phi,E})}{2\pi i}}\right)ds=0
\label{efatopi.107}\end{multline}
since $\Tr( e^{\frac{w(0)}{2\pi i}}-e^{\frac{w(1)}{2\pi i}})=0$.

\begin{lemma}\label{efatopi.173} Under the inversion map
  $I:G^{-\infty}(\phi,E)\ni\sigma\longmapsto\sigma^{-1}\in G^{-\infty}(\phi;E)$ the
  Chern character pulls back to its negative $I^*\Ch_{\odd}=-\Ch_{\odd}.$
\end{lemma}

\begin{proof} This follows directly from \eqref{gcf.7} since 
\begin{equation}
w(s,\sigma ^{-1},\nabla)=\sigma w(1-s,\sigma,\nabla)\sigma^{-1},\
\sigma\nabla^{\phi,E}\sigma^{-1}=-(\nabla^{\phi,E}\sigma)\sigma ^{-1}
\label{efatopi.174}\end{equation}
and the conjugation invariance of the trace.
\end{proof}

Furthermore, the cohomology class defined by $\Ch_{\odd}$ is additive, in
the sense that
\begin{equation}
\Ch_{\odd}(\sigma _1\sigma _2)=\Ch_{\odd}(\sigma _1)+\Ch_{\odd}(\sigma _2)+dF.
\label{efatopi.101}\end{equation}
Again it is useful to give a universal version of such a multiplicativity
formula.

Let $(G^{-\infty}(\phi,E))^{[2]}$ be the fibre product of
$G^{-\infty}(\phi,E)$ with itself as a bundle over $Y.$ Then there are
the usual three maps 
\begin{equation}
\xymatrix{
&G^{-\infty}(\phi,E)\\
&(G^{-\infty}(\phi,E))^{[2]}\ar[dl]_{\pi_S}\ar[dr]^{\pi_F}\ar[u]_{m}\\
G^{-\infty}(\phi,E)&&G^{-\infty}(\phi,E),
}
\label{efatopi.102}\end{equation}
where $\pi_F(a,b)=b,$ $\pi_S(a,b)=a$ and $m(a,b)=ab.$

\begin{proposition}\label{efatopi.103} There is a smooth form $\mu$ on
  $(G^{-\infty}(\phi,E))^{[2]}$ such that 
\begin{equation}
m^*\Ch_{\odd}=\pi_F^*\Ch_{\odd}+\pi_S^*\Ch_{\odd}+d\mu.
\label{efatopi.104}\end{equation}
\end{proposition}

\begin{proof} This follows from essentially the same argument as used in
  the proof of Proposition~\ref{efatopi.140}. Thus, consider the bundle of
  groups with $E$ replaced by $E\oplus E$ in which the original is embedded
  as acting on the first copy and as the identity on the second copy. As in
  \eqref{efatopi.158}, this action on $E\oplus E$ is homotopic under a
  family of $2\times2$ absolute rotations, i\@.e\@.~not depending on the
  space variables, to the action on the second copy with the identity on
  the first. Now, $m^*\Ch_{\odd}$ is realized through the product,
  i\@.e\@.~diagonal action on the first factor. Applying the rotations but
  just in the second term of this product action, the map $m$ is
  homotopic to $\pi_S\oplus\pi_F.$ Since the Chern character is closed, the
  same argument as in Proposition~\ref{efatopi.140} constructs the
  transgression form $\mu$ as the integral of the variation along the
  homotopy.
\end{proof}

\section{Odd elliptic families}\label{Odd-ell}

Now we turn to the consideration of a given family of self-adjoint elliptic
pseudodifferential operators, $A\in\Psi^1(M/Y;E).$ In fact it is not
self-adjointness that we need here, but rather the consequence that $A+it,$
which is a product-type family in the space $\Psi^{1,1}_{\psus}(M/Y;E),$
should be fully elliptic and hence invertible for large real $t,$ $|t|>T.$
See  \cite{arXiv:math/0606382} and  \cite{fipomb2} for a discussion of product-suspended pseudodifferential
operators. More generally we may simply start with an elliptic family in
this sense, possibly of different order, $A\in\Psi^{m,l}_{\psus}(M/Y;E).$
It follows from the assumed full ellipticity that for each value of the
parameter $y\in Y$ the set of invertible perturbations
\begin{equation}
\begin{gathered}
\begin{aligned}
\cA_y=\big\{A+it+q(t);&q\in\Psi^{-\infty}_{\sus}(Z_y;E_y),\\
&(A+it+q(t))^{-1}\in\Psi^{-1}(Z_y;E_y)\ \forall\ t\in\bbR\big\}\Mor
\end{aligned}
\\
\cA_y=\left\{A(t)+q(t);q\in\Psi^{-\infty}_{\sus}(Z_y;E_y),\
(A(t)+q(t))^{-1}\in\Psi^{-m}(Z_y;E_y)\ \forall\ t\in\bbR\right\}
\end{gathered}
\label{efatopi.82}\end{equation}
is non-empty. This is discussed in the proof below.

\begin{proposition}\label{efatopi.83} If $A\in\Psi^1(M/Y;E)$ is an elliptic
  and self-adjoint family or $A\in\Psi^{m,l}_{\ps}(M/Y;E)$ is a fully
  elliptic product-type family then \eqref{efatopi.82} defines a smooth (infinite
  dimensional) Fr\'echet subbundle
  $\cA(\phi)\subset\Psi^{m,l}_{\psus}(\phi,E)$ (where $m=l=1$ in the
  standard case) over $Y$ with fibres which are principal spaces for the
  action of the bundle of groups $G^{-\infty}_{\sus}(\phi;E).$
\end{proposition}

\begin{proof} The non-emptiness of the fibre at any point follows from
  standard results for the even index. Namely at each point in the base,
  the family $A_y(t)$ is elliptic and invertible for large $|t|$ as a
  consequence of the assumed full ellipticity. Thus the index of this
  family is an element of $\Kce(\bbR)$ and hence vanishes. For such a
  family there is a compactly supported, in the parameter, family $q(t)$ of
  smoothing operators on the fibre which is such that $A(t)+q(t)$ is
  invertible for all $t\in\bbR.$

Now the fact that the fibre is a principal space for the group
$G^{-\infty}_{\sus}(Z_y;E_y)$ follows directly, since for two such
perturbations $q_i,$ $i=1,2,$  
\begin{equation}
A(t)+q_1(t)=(\Id_{E_y}+q_{12}(t))(A(t)+q_2(t)),\
q_{12}\in\Psi^{-\infty}_{\sus}(Z_y;E_y)
\label{efatopi.90}\end{equation}
and conversely. The local triviality of this bundle follows from the fact
that invertibility persists under small perturbations.
\end{proof}

It is this bundle, $\cA(\phi),$ which we think of as \emph{the} index bundle
since the existence of a global section is equivalent to the vanishing, in odd
K-theory, of the index of the original family. Note that since we permit
the orders $m$ and $l$ of a fully elliptic family in
$\Psi^{m,l}_{\psus}(\phi,E)$ to take values other than $1,$ the index
is additive in the sense that the index bundles of two elliptic
families (acting on the same bundle) compose in the obvious way.

\section{Eta forms for an odd family}\label{EtaFam}

To define the geometric eta form, recall that it is shown above that a
covariant derivative is induced on the bundle
$\Psi^{m}(\phi;E)\longrightarrow Y$ from the connection on $\phi$ and the
connection on $E.$ Consider the subbundle of elliptic and invertible
pseudodifferential operators, $G^{m}(\phi;E).$ Since product-suspended
operators can be seen as one-parameter families of pseudodifferential
operators, there is an evaluation map
\begin{equation}
\ev:\bbR_{\tau} \times \cA(\phi)\ni(\tau,a)\longmapsto a(\tau)\in G^{m}(\phi;E)
\label{gef.1}\end{equation}
compatible with the bundle structure. On $G^{m}(\phi;E),$ consider as
before the tautological bundle
\begin{equation}
\pi^{*}G^{m}(\phi;E)\longrightarrow G^{m}(\phi;E)
\label{gef.2}\end{equation}
obtained by pulling back the bundle to its own total space. With $a:
G^m(\phi;E)\longrightarrow \pi^{*}G^m(\phi;E)$ the
tautological section consider the odd form
\begin{equation}
\lambda= \frac{1}{2\pi i} a^{-1}\tilde\nabla a \int_{0}^{1} 
\exp\left( \frac{s(1-s)(a^{-1}\tilde\nabla a)(a^{-1}\tilde\nabla a)
+(s-1)\tilde\omega-s a^{-1}\tilde\omega a}{2\pi i} \right)ds
\label{gef.3}\end{equation}
taking values in $\pi^{*}\Psi^{0}(\phi;E).$ Here $\tilde\nabla=
\pi^{*}\nabla $ and $\tilde{\omega}= \pi^{*}\omega$ with $\omega$
defined in \eqref{gcf.3}. 

This is \emph{formally} the same as the argument of the trace functional in
\eqref{gcf.10}, except of course that now the section $a$ is no longer a
perturbation of the identity by a smoothing operator, but an invertible operator of order $m$. Nevertheless, the
identities in \eqref{efatopi.105} still hold, since they are based on the
Bianchi identity. Writing the pull-back under the evaluation map as
\begin{equation}
\ev^{*}(\lambda)=\lambda^{t}+\lambda^{n}\wedge
d\tau,\ \lambda^{n}=\iota_{\pa \tau}\lambda
\label{gef.4}\end{equation}
both the tangential and normal parts are form-valued sections of the bundle
\begin{equation}
   \pi^{*}_{\cA}\Psi^{0,0}_{\ps}(\phi;E)\longrightarrow \cA
\label{gef.6}\end{equation}
obtained by pulling back 
$\pi:\Psi^{0,0}_{\ps}(\phi;E)\longrightarrow  Y$ to $\cA.$     

\begin{definition}\label{gef.5} On the total space of the bundle $\cA$ the
(even) \emph{geometric eta form} is
\begin{equation}
   \eta_{\cA}= \Tr_{\sus}(\lambda^{n})
\label{efatopi.47}\end{equation}
where $\Tr_{\sus}$ is the regularized trace of \cite{MR96h:58169} taken fibrewise
in the fibres of the bundle \eqref{gef.6}.  
\end{definition}

The bundle~\eqref{gef.6} does not have a tautological section, but
\begin{equation}
\pi^{*}_{\cA}G^{m,l}_{\ps}(\phi;E)\longrightarrow \cA
\label{gef.7}\end{equation}
does, where $G^{m,l}_{\ps}(\phi;E)\subset\Psi^{m,l}_{\ps}(\phi;E)$ is the
subbundle of elliptic invertible elements, with product-type
pseudodifferential inverses. Denote this section by
$\alpha_{\cA}: \cA\longrightarrow \pi_{\cA}^{*}G^{m,l}_{\ps}(\phi;E).$
Then consider also the odd form 
\begin{equation}
\tilde\gamma_{\cA}=
\frac{1}{2\pi i} \widetilde\Tr \left[
 \int_{0}^{1} \alpha_{\cA}^{-1}\hat\nabla\alpha_{\cA} \exp\left( 
 \frac{ s(1-s)(\alpha_{\cA}^{-1} \hat\nabla \alpha_{\cA})^{2}+ (s-1)\hat\omega -s 
 \alpha_{\cA}^{-1}\hat\omega \alpha_{\cA}}{2\pi i} \right) ds\right]
\label{gef.8}\end{equation}
where $\widetilde\Tr$ is the \emph{formal trace} from \cite{MR96h:58169}
and $\hat{\nabla}= \pi_{\cA}^{*}\nabla$, $\hat{\omega}= \pi^{*}_{\cA} \omega$.

\begin{proposition}\label{gef.10} For an odd elliptic family of first order
  (so either self-adjoint or directly of product type), the exterior derivative of
  the geometric eta form
\begin{equation}
d\eta_{\cA}=\tilde\gamma_{\cA}=\pi^{*}_{\cA}\gamma_{A}
\label{efatopi.48}\end{equation}
is the pull-back of a closed form on the base $\gamma_{A}\in\CI(Y;\Lambda
^{\odd}).$
\end{proposition}

\begin{proof}
Recall first that the regularized trace is defined in \cite{MR96h:58169} by
taking the  constant term in the asymptotic expansion of
\begin{equation}
\int_{-\tau}^{\tau} \int_{0}^{\tau_{p}}\cdots \int_{0}^{\tau_{1}} 
\Tr\left(\frac{\pa^{p}}{\pa r^{p}} \lambda^{t}(r)\right)dr d\tau_{1}\cdots
 d\tau_{p}
\label{gef.11}\end{equation}
as $\tau\to +\infty.$ Here $p\in \bbN$ is chosen large enough so that 
$\frac{\pa^{p}}{\pa\tau^{p}} \lambda^{t}(\tau)$ is of trace class -- the
product-suspended property implies that high $\tau$ derivatives are of
correspondingly low order in both senses. This is a trace,
i\@.e\@.~vanishes on commutators, but is not exact in the sense that
$\widetilde\Tr(A)=\Tr_{\sus}(\frac{\pa A}{\pa\tau})$ does not necessarily
vanish, but \emph{is} determined by the asymptotic expansions of $A(\tau)$ as
$\pm\tau\to\infty$ since it does vanish for smoothing Schwartz
perturbations of $A.$ 

As noted above, the identities \eqref{efatopi.105} hold for $\lambda.$ For
the pull-back under the evaluation map this means that modulo commutators 
\begin{equation}
\tilde \nabla\lambda ^n\equiv\frac{\pa}{\pa \tau}\lambda^t.
\label{efatopi.108}\end{equation}
Now, taking $p$ further derivatives with respect to $\tau$ gives the same
identity modulo commutators where the sum of the orders of the terms
becomes low as $p$ increase. So, applying the trace functional the
commutator terms vanish and it follows that
\begin{equation}
\Tr\left(\frac{\pa^{p}}{\pa\tau^{p}}
\tilde{\nabla}\lambda^{n}(\tau)\right)\wedge d\tau =
\Tr\left(\frac{\pa^{p}}{\pa\tau^{p}} \frac{\pa}{\pa\tau}
\lambda^{t}(r)\right)\wedge d\tau.
 \label{gef.13}\end{equation}
Therefore,
 \begin{equation}
 \begin{aligned}
d\eta_{\cA} &=
\Tr_{\sus}( \tilde{\nabla} \lambda^{n})= \Tr_{\sus}\left(\frac{\pa}{\pa\tau} 
     \lambda^{t}\right) \\
&= \widetilde\Tr (\lambda^{t}) =\tilde\gamma_{\cA}
  \end{aligned}           
 \label{gef.14}\end{equation}
by definition of the formal trace.

As already noted, the formal trace vanishes on low order perturbations so
$\tilde\gamma_{\cA}$ is basic, i\@.e\@.~is actually the pull-back of a
well-defined form $\gamma_{A}$ on $Y$ depending only on the initial family
$A.$
\end{proof}

The index formula \eqref{efatopi.5} therefore amounts to showing that the
form $\gamma_{A}$ represents the (odd) Chern character of the index of the
family $A$ in cohomology. This is difficult to approach directly,
computationally, so instead we show how the index bundle $\cA(\phi)$ can be
`trivialized' by extending the bundle of structure groups.

If $A\in\Psi^{m,l}_{\psus}(M/Y;E)$ is a fully elliptic family then it has
an `inverse' family which, whilst not completely well-defined, is determined
up to smoothing terms. Namely the bundle $\cA$ is locally trivial over $Y$
and in particular has local sections. Taking a partition of unity $\psi_i$ on $Y$
subordinate to a cover by open set $U_i$ over each of which there is a
section $A_i$ the inverse family can be taken to be 
\begin{equation}
\sum\limits_{i}\psi_{i}(y)A_i^{-1}\in\Psi^{-m,-l}_{\psus}(M/Y;E).
\label{efatopi.171}\end{equation}
It is fully elliptic and, essentially by definition, the corresponding
bundle of invertible perturbations is naturally identified with the bundle
$\cA^{-1}(\phi)\subset\Psi^{-m,-l}_{\psus}(\phi;E)$ consisting of the
inverses of the elements of $\cA(\phi).$

\begin{lemma}\label{efatopi.172} Under the inversion map
  $\cA(\phi)\longrightarrow \cA^{-1}(\phi)$ the eta form $\eta_{\cA^{-1}}$ on
  $\cA^{-1}(\phi)$ associated with the inverse family \eqref{efatopi.171}
  pulls back to $-\eta_{\cA}.$ 
\end{lemma}

\begin{proof} The proof is similar to the one of  Lemma~\ref{efatopi.173},
namely since the regularized trace vanishes on commutators, the result
follows from the analog of \eqref{efatopi.174} for the bundles $\cA(\phi)$ and
$\cA^{-1}(\phi)$. 
\end{proof}

We also need a variant of the multiplicative formula
\eqref{efatopi.104}. For this, consider the fibre product $\cA^{[2]}(\phi)$ of two copies of
the bundle $\cA(\phi)$  and the product map given by
inversion in the second map 
\begin{equation}
\tilde m:\cA^{[2]}(\phi)\ni(A',A)\longmapsto A'A^{-1}\in G^{-\infty}_{\sus}(\phi;E).
\label{efatopi.175}\end{equation}

\begin{proposition}\label{efatopi.176} Under pull-back under the three maps 
\begin{equation}
\xymatrix{
&G^{-\infty}_{\sus}(\phi,E)\\
&\cA^{[2]}(\phi)\ar[dl]_{\pi_S}\ar[dr]^{\pi_F}\ar[u]_{\tilde m}\\
\cA(\phi)&&\cA(\phi),
}
\label{efatopi.177}\end{equation}
\begin{equation}
\tilde m^*\Ch_{\even}-\pi_S^*\eta_{\cA}+\pi_F^*\eta_{\cA}=d\delta _{\cA}
\label{efatopi.178}\end{equation}
for a smooth form $\delta _{\cA}$ on $\cA^{[2]}(\phi).$
\end{proposition}

\begin{proof} We perform the same deformation as in
  Proposition~\ref{efatopi.103} and its earlier variants. Taking into
  account Proposition~\ref{gef.10}  and Lemma~\ref{efatopi.173}, it follows that the same conclusion holds
  except that extra terms may appear from $\tau\to\pm\infty.$ These give a
  basic form so in place of the desired identity \eqref{efatopi.178} we
  find instead that 
\begin{equation}
\tilde m^*\Ch_{\even}-\pi_S^*\eta_{\cA}+\pi_F^*\eta_{\cA}=d\delta' _{\cA}+\pi^*\mu
\label{efatopi.179}\end{equation}
where $\delta '_{\cA}$ is smooth form on $\cA^{[2]}(\phi)$ and $\mu$ is
smooth form on the base, with $\pi:\cA^{[2]}(\phi)\longrightarrow Y.$
However, under exchange of the two factors, the left side of
\eqref{efatopi.179} changes sign, while the final, basic, term is
unchanged. Thus taking the odd part of \eqref{efatopi.179} gives
\eqref{efatopi.178}. 
\end{proof}
It is also of interest to see how the eta form transforms under a change of connections for the fibration $\phi:M\to Y$ and the bundle $E\to M$.  
\begin{lemma}
If $\eta_{\cA}$ and $\eta_{\cA}'$ are eta forms associated to the self-adjoint elliptic family $A\in \Psi^{1}(M/Y;E)$ with respect to two different choices of connections for the fibration $\phi:M\to Y$ and the vector bundle $E\to M$, then there exist a form $\beta\in \Omega^{\even}(Y)$ and a form
$\alpha\in \Omega^{\odd}(\cA)$ such that 
\[
      \eta_{\cA}'-\eta_{\cA}= \pi^{*}_{\cA}\beta +d\alpha.  
\]  
\label{coc.1}\end{lemma}
\begin{proof}
Consider the new fibration $\phi\times \Id: M\times [0,1]_{t}\to Y\times [0,1]_{t}$ with family $\pr^{*}A$ where $\pr: M\times [0,1]\to M$ is the projection on the left factor.  Clearly, we can choose  connections for the fibration $\phi\times \Id$ and the bundle $\pr^{*}E\to M\times [0,1]_{t}$ such that if $\eta_{\pr^{*}\cA}$ is the corresponding eta form, then 
\[
    \iota_{0}^{*}\eta_{\pr^{*}\cA}= \eta_{\cA}, \quad  \iota_{1}^{*}\eta_{\pr^{*}\cA}= \eta_{\cA}',
\]
 where for $i\in [0,1]$, $\iota_{i}: \cA\to \cA\times\{i\}\subset \cA\times [0,1]_{t}\cong\pr^{*}\cA$ is the natural inclusion.  The form $\eta_{\pr^{*}\cA}$ can be written as
 \begin{equation}
   \eta_{\pr^{*}\cA}= \pr_{\cA}^{*}\omega_{0}(t)+ \pr^{*}_{\cA}\omega_{1}(t)\wedge dt, \quad 
   \omega_{0}(t)\in \Omega^{\even}(\cA), \; \omega_{1}(t)\in \Omega^{\odd}(\cA),
 \label{coc.2}\end{equation}
where $\pr_{\cA}: \pr^{*}\cA\cong \cA\times [0,1]\to \cA$ is the natural projection.  This suggests to define $\alpha\in \Omega^{\even}(\cA)$ by 
\[
    \alpha= (\pr_{\cA})_{*}(\eta_{\pr^{*}\cA})= \int_{0}^{1}\omega_{1}(t)dt.
\]
Then we have that 
\begin{equation}
\begin{aligned}
d\alpha &= (\pr_{\cA})_{*}(d\eta_{\pr^{*}\cA})+ \eta_{\cA}'-\eta_{\cA} \\
              &= \pi^{*}_{\cA} \pr_{*}(\gamma_{\pr^{*}\cA})+ \eta_{\cA}' -\eta_{\cA}.
\end{aligned}
\label{coc.3}\end{equation}
Thus, the result follows by taking $\beta= -\pr_{*}(\gamma_{\pr^{*}\cA})= -\int_{0}^{1}v_{1}(t)dt$ where 
\[
\gamma_{\pr^{*}\cA}= \pr^{*}v_{0}(t)+ \pr^{*}v_{1}(t)\wedge dt, \quad v_{0}(t)\in \Omega^{\odd}(Y), \;
v_{1}(t)\in \Omega^{\even}(Y),
\]
and where $\pr$ also denotes the natural projection $\pr: Y\times [0,1]_{t}\to Y$.

\end{proof}

\section{Extended eta invariant}\label{Exteta}

The bundle $\cA(\phi)$ in \eqref{efatopi.82} is a bundle of principal
spaces for the action of the fibres of $G^{-\infty}_{\sus}(\phi;E).$ The
fibres can be enlarged to give an action of the central, contractible,
group in \eqref{efatopi.8} by setting
\begin{equation}
\tcA_{y}=\tilde G^{-\infty}(Z_y;E)\cdot\cA_{y}.
\label{pr.1}\end{equation}
This is not so easily characterized additively but is the image of the
quotient map on the fibre product
\begin{equation}
p_{\sim}:\widetilde{G}^{-\infty}_{\sus}(\phi;E)\times_Y\cA(\phi)\longrightarrow
\tcA(\phi),\ p_{\sim}(\tilde g,A)=\tilde g A.
\label{efatopi.97}\end{equation}
In particular there is an exact and fibrewise delooping sequence coming from
\eqref{efatopi.8}: 
\begin{equation}\xymatrix{
\cA(\phi)\ar[r]&\tcA(\phi)\ar[r]^-{\tilde R_{\infty}}&
\widetilde{G}^{-\infty}_{\sus}(\phi;E).}
\label{efatopi.163}\end{equation}
The quotient map here can be defined in the fibre $\tcA_{y}$ by
\begin{equation}
\tilde R_{\infty}(A'_{y})=
\lim_{\tau\to\infty}A'_{y}A_{y}^{-1}
\label{pr.4}\end{equation}
where $A_{y}\in\cA_y$ is any point in the fibre of $\cA(\phi)$ over
the same basepoint. Clearly the result does not depend on this choice of $A_{y}.$

The construction above of the eta form on $\cA(\phi)$ extends to
$\tcA(\phi).$ Thus, the same form \eqref{gef.3} pulls back under the
evaluation map 
\begin{equation}
\widetilde{\ev}:\bbR_{\tau}\times\tcA(\phi)\longrightarrow G^m(\phi;E)
\label{efatopi.91}\end{equation}
to give 
\begin{equation}
{\tev}^*(\lambda)=\widetilde{\lambda}^t+\widetilde{\lambda
}^n\wedge d\tau.
\label{efatopi.92}\end{equation}
Then, extending Definition~\ref{gef.5}, set 
\begin{equation}
\eta_{\tcA}=\Tr_{\sus}(\widetilde{\lambda}^{n}).
\label{efatopi.93}\end{equation}
We use this extended bundle and eta form to analyse the
invariance properties of $\eta_{\cA}.$

Consider the fibre product with projections and quotient map
\begin{equation}
\xymatrix{
&\tcA(\phi)\ar[dr]^{\tilde{R}_{\infty}}\\
&\widetilde{G}^{-\infty}_{\sus}(\phi;E)\times_Y\cA(\phi)\ar[dl]_{\pi_{\cA}}
\ar[dr]^{\pi_{\tilde G}}\ar[u]^{p_{\sim}}\ar[r]^-{R}
&G^{-\infty}(\phi;E)\\
\cA(\phi)&&\widetilde{G}^{-\infty}_{\sus}(\phi;E)\ar[u]^-{R_{\infty}}
}
\label{efatopi.96}\end{equation}
\begin{proposition}\label{efatopi.99} The diagram
  \eqref{efatopi.96} commutes and there are smooth forms
  $\tilde\delta_{\cA}$ and $\tilde\mu_{\cA}$ respectively on the fibre product and
  $G^{-\infty}(\phi;E)$ such that the three eta forms pull back to satisfy
\begin{equation}
p_{\sim}^*\eta_{\tcA}=
\pi^*_{\cA}\eta_{\cA}+\pi_{\tilde G}^*\tilde\eta+d\tilde\delta _A+R^*\tilde\mu_{\cA}.
\label{efatopi.100}\end{equation}
\end{proposition}

\begin{proof} The commutativity of the parallelogram on the right is
  discussed above and defines the diagonal map, $R.$ 

The formula \eqref{efatopi.100} is a generalization of that of
Proposition~\ref{efatopi.103} and the proof proceeds along the same lines.
Consider the odd Chern character on $E\oplus E.$ 
Thus, from the fibre product there are two evaluation maps, $\tEv$ and
$\tev$ and we may combine these using the bundle rotation as in
\eqref{efatopi.156}. This gives the two-parameter family of maps from
$\widetilde{G}^{-\infty}_{\sus}(\phi;E)\times_Y\cA(\phi):$ 
\begin{multline}
[0,1]_t\times\bbR\times \widetilde{G}^{-\infty}_{\sus}(\phi;E)\times_Y\cA(\phi)\ni
(t,\tau,\tilde g,\widetilde{A})\longmapsto\\
M^{-1}(t)\begin{pmatrix}\tilde g_y(\tau)&0\\0&\Id
\end{pmatrix}M(t)\begin{pmatrix}\Id&0\\0&\widetilde{A}_y(\tau)
\end{pmatrix}\in
G^{m}(\phi,E\oplus E).
\label{efatopi.166}\end{multline}

Pulling back the form $\lambda$ of \eqref{gef.3} under this map and
`integrating' over $[0,1]_t\times\bbR_{\tau}$ gives the identity
\eqref{efatopi.100}, where the $\tau$ integral is to be interpreted as part
of the regularized trace. Since the form $\lambda$  is closed
modulo commutators, if the product decomposition of its pull-back is
\begin{equation}
dt\wedge d\tau\wedge\mu +dt\wedge\lambda_t +d\tau\wedge\lambda_{\tau} +\lambda '
\label{efatopi.167}\end{equation}
then 
\begin{equation}
d\mu-\frac{\pa \lambda _t}{\pa\tau}+\frac{\pa \lambda _{\tau}}{\pa t}\equiv0
\label{efatopi.168}\end{equation}
again modulo commutators. The regularized trace and integral of the last
term gives the difference of the three pulled-back eta forms and $\mu$
defines the term $\tilde\delta_{\cA}$ on the fibre product.

Thus it remains to analyse the second term in \eqref{efatopi.168}. The
exterior differentials in \eqref{gcf.10} each fall on either a factor from
$\tilde G^{-\infty}_{\sus}(\phi;E)$ or on $\cA.$ The terms involving no
derivative of the first type, so the `pure $\cA$ part', makes no
contribution, since as discussed earlier, the rotation factor $M(t)$
disappears. Thus, only terms with at least one derivative falling on the
first three factors of \eqref{efatopi.166} need to be considered. This results in a smoothing
operator and the regularization of the trace functional is not
necessary. Then the $\tau$ integral reduces to the value of $\lambda_{t}$ at
$\tau=\infty$ which depends only on the leading term in $\cA$ as
$\tau\to\infty,$ which is to say the corresponding term in $A$ itself, and
$R_{\infty}(\tilde g).$ This leads to the additional term $\tilde{\mu} _A$ in 
\eqref{efatopi.100}.
\end{proof}

The left action of the groups on the fibres of the index bundle induces a
contraction map 
\begin{equation}
\begin{gathered}
L:G^{-\infty}_{\sus}(\phi;E)\times_Y\cA(\phi)\longrightarrow \cA(\phi)\\
L(s,u)= su,\ \forall\ s\in G^{-\infty}(Z_y;E_y),\ u\in \cA_{y}\ \forall\ y\in Y.
\end{gathered}
\label{prim.2}\end{equation}

\begin{corollary}\label{prim.1} There is a smooth odd form
$\tilde\delta$ on $G^{-\infty}_{\sus}(\phi;E)\times_Y\cA(\phi)$
such that
\begin{equation}
L^{*}\eta_{\cA}=\pi_{\cA}^*\eta_{\cA}+\pi^{*}_{G}\Ch_{\even}+d\tilde\delta.
\label{efatopi.85}\end{equation}
\end{corollary}

\begin{proof} Restricting to the subbundle $G_{\sus}^{-\infty}(\phi;E)$ in
  \eqref{efatopi.96} gives a diagram which includes into it and on which
  $R_{\infty}$ and $R$ are trivial: 
\begin{equation}
\xymatrix{
&\cA(\phi)\\
&{G}^{-\infty}_{\sus}(\phi;E)\times_Y\cA(\phi)\ar[dl]_{\pi_{\cA}}
\ar[dr]^{\pi_{G}}\ar[u]^{p}\\
\cA(\phi)&&{G}^{-\infty}_{\sus}(\phi;E).
}
\label{efatopi.164}\end{equation}
Thus \eqref{efatopi.100} restricts to this diagram with $\tilde\mu _A$
vanishing and $\tilde\eta$ reducing to $\Ch_{\even}$ on
$G^{-\infty}_{\sus}(\phi;E).$ Thus \eqref{efatopi.85} follows.
\end{proof}

\begin{remark}
If $s:U\longrightarrow G^{-\infty}_{\sus}(\phi;E)$ and $\alpha:U\longrightarrow\cA$
are sections over an open subset $U\subset Y,$ then corollary~\ref{prim.1}
shows that  
\begin{equation}
\eta(s\alpha)-\eta(\alpha)= \Ch_{\even}(s)+d \alpha^{*}\tilde\delta.
\label{efatopi.87}\end{equation}
\label{prim.3}\end{remark}

\section{Index formula: Proof of Theorem~\ref{efatopi.4}}\label{Index-form}

The bundle $\tcA(\phi)$ has contractible fibres and hence has a
  global continuous section $\tilde A:Y\longrightarrow \tcA(\phi);$ this
  section is easily made smooth. The inverse image
  of the range of this section under the vertical map, $p_{\sim},$ in
  \eqref{efatopi.96} is a submanifold
  $\cF\subset\widetilde{G}^{-\infty}_{\sus}(\phi;E)\times_Y\cA(\phi).$
  Indeed for each $y\in Y$ and each $B_y\in\cA_y$ there is a
  unique $Q_y\in\widetilde{G}^{-\infty}_{\sus}(Z_y;E)$ such that 
\begin{equation*}
Q(\tau)B_y=\tilde A_y(\tau)
\label{efatopi.109}\end{equation*}
is the value of the section at that point. Thus, $\pi_{\cA}$ restricts to an
isomorphism from $\cF$ to $\cA.$

Using the section $\tilde{\cA}: Y\longrightarrow \tilde{\cA}(\phi),$ we can
identify $p_{\sim}(\cF)$ with $Y$ so that restricting \eqref{efatopi.96} to
$\cF$ gives the commutative diagram
\begin{equation}
\xymatrix{
&Y\ar[dr]^-{\gamma}\\
&\cF\ar[dl]_{\pi_{\cA}}^{\simeq}
\ar[u]^{\pi}\ar[r]^-{\widetilde\gamma}\ar[dr]_-{\pi_{\tilde G}}
&G^{-\infty}(\phi;E)\\
\cA(\phi)&&\widetilde{G}^{-\infty}_{\sus}(\phi;E)\ar[u]_-{R_{\infty}}
}
\label{efatopi.110}\end{equation}
where $\widetilde\gamma$ is the restriction of $R$ to $\cF$ and $\gamma$ is the
classifying map defined to make this diagram commutes.

Restricted to $\cF$ the identity \eqref{efatopi.100} becomes modulo exact 
forms 
\begin{equation}
\pi_{\cA}^*\eta_{\cA}+\pi^*_{\tilde
  G}\tilde\eta=\pi^*\beta_{\tilde A},\
\beta_{\tilde A}=\tilde A^*\widetilde\eta_{\tilde
 A}-\gamma^*\tilde{\mu}_{\cA} \in\CI(Y;\Lambda ^{\even}).
\label{efatopi.111}\end{equation}
From \eqref{efatopi.22} 
\begin{equation}
\pi^*_{\tilde G}d\tilde\eta=\widetilde\gamma^*(\Ch_{\odd})
\label{efatopi.112}\end{equation}
so pulling back to $\cA$ under the isomorphism $\pi_{\cA}$ gives the index
formula \eqref{efatopi.5}: 
\begin{equation}
d\eta_{\cA}=-\pi^*\gamma^*(\Ch_{\odd})+d\beta_{\tilde A}.
\label{efatopi.113}\end{equation}
Since the homotopy class of the section $\gamma$ represents \textbf{minus}
the index class, 
this shows that $\gamma_A$ in \eqref{efatopi.48}
represents the Chern character of the index.

\section{Determinant of an odd elliptic family}\label{Dets}

The eta invariant, interpreted here as the degree zero part in the eta form
(there is a factor of $2$ compared to the original normalization of Atiyah,
Patodi and Singer) is a normalized log-determinant. In the universal case, for the
classifying spaces, $\tilde\eta^0$ is a well-defined function on $\tilde
G^{-\infty}_{\sus}(Z;E)$ and then 
\begin{equation}
\det(g)=\exp(2\pi i\tilde\eta^0)
\text{ is the Freholm determinant on }G^{-\infty}(Z;E).
\label{efatopi.50}\end{equation}

In the geometric case essentially the same result is true.

\begin{proposition}\label{efatopi.51} For $A\in\Psi^{m,k}_{\psus}(M/Y;E)$ a
fully elliptic family of product-type operators on the fibres of a
fibration,
\begin{equation}
\tau(A)=\exp(2\pi i\eta_{\cA}^0)\in\CI(Y;\bbC^*),
\label{efatopi.52}\end{equation}
where $\eta_{\cA}^0$ is the degree zero part in \eqref{efatopi.47}, is a
multiplicative function on fully elliptic operators on a fixed bundle,
\begin{equation}
\tau(AB)=\tau(A)\tau(B),\ A\in\Psi^{m,k}_{\psus}(M/Y;E),
B\in\Psi^{m',k'}_{\psus}(M/Y;E)
\label{efatopi.169}\end{equation}
which is constant under smoothing perturbation and which represents the
class associated to $\ind(A)\in\Ko(B)$ in $\operatorname{H}^1(Y;\bbZ).$
\end{proposition}

\begin{proof} Theorem~\ref{efatopi.4} shows that $d\eta_{\cA}^0$ defined in
  principle on $\cA,$ the bundle of invertible perturbations of a given
  fully elliptic family $A,$ is basic and represents the first odd Chern
  class of the index. For the zero form part, \eqref{efatopi.85} implies
  true multiplicativity under the action of $G^{-\infty}_{\sus}(\phi;E),$
  with the zero form part of $\Ch_{\even}$ being the numerical index. Thus
  indeed the tau invariant in \eqref{efatopi.52} is a well defined function
  on the base which represents the first odd Chern class in integral
  cohomology.

Full multiplicativity follows as in \cite{MR96h:58169}.
\end{proof}

\section{Doubly suspended determinant}\label{DoubDet}

Let $G^{-\infty}_{\sus(2)}(Z;E)$ be the double flat-smooth loop
group. Thus, its elements are Schwartz functions $a:\bbR^2\longrightarrow
\Psi^{-\infty}(Z;E)$ such that $\Id+a(t,\tau)$ is invertible for each
$(t,\tau)\in\bbR^2.$ Let
\begin{equation}
G^{-\infty}_{\sus(2)}(Z;E)[\epsilon /\epsilon^2]=
G^{-\infty}_{\sus(2)}(Z;E)\oplus \cS(\bbR^2;\Psi^{-\infty}(Z;E))
\label{efatopi.53}\end{equation}
be the group with the truncated $*$ (or Moyal) product obtained as in
\cite{arXiv:math/0606382} by adiabatic limit from the isotropic product on
$\cS(\bbR^2;\Psi^{-\infty}(Z;E))$ and then passing to the quotient by terms of
order $\epsilon ^2.$ Explicitly this product is
\begin{equation}
(a_0+\epsilon a_1)[*](b_0+\epsilon b_1)=a_0b_0+\epsilon
  \left(a_0b_1+a_1b_0-\frac{1}{2i}\left(\frac{\pa a_0}{\pa t}\frac{\pa b_0}{\pa
    \tau}-\frac{\pa a_0}{\pa\tau}\frac{\pa b_0}{\pa t}\right)\right)
\label{efatopi.119}\end{equation}
where the underlying product is in $\Psi^{-\infty}(Z;E).$

As shown in \cite{arXiv:math/0606382}, in the adiabatic
limit the Fredholm determinant, for operators on $Z\times\bbR,$ induces the
`adiabatic determinant'
\begin{equation}
{\det}_{\ad}:G^{-\infty}_{\sus(2)}(Z;E)[\epsilon /\epsilon ^2]\longrightarrow \bbC^*,\
{\det}_{\ad}(g_1g_2)={\det}_{\ad}(g_1){\det}_{\ad}(g_2),
\label{efatopi.54}\end{equation}
which, as in the unsuspended case, generates the $1$-dimensional integral
cohomology of $G^{-\infty}_{\sus(2)}(Z;E)[\epsilon /\epsilon^2]$ -- which
is classifying for odd K-theory. Note that there is no such multliplicative
function on the leading group, without the first order (in $\epsilon)$
`correction' terms in \eqref{efatopi.119}.

To define the adiabatic determinant, one needs to consider the adiabatic trace
on $\Psi^{-\infty}_{\sus(2)}(Z;E)[\epsilon/\epsilon^{2}]$ defined by
\begin{equation}
\Tr_{\ad}(a)= \frac{1}{2\pi} \int_{\bbR^{2}} \Tr_{Z}( a_{1}(t,\tau))dt d\tau, 
\ a= a_{0}+\epsilon a_{1} \in \Psi^{-\infty}_{\sus(2)}(Z;E)[\epsilon/\epsilon^{2}]
\label{td.1}\end{equation}
with $a_{0},a_{1}\in \Psi^{-\infty}_{\sus(2)}(Z;E).$ A special case of
Lemma~\ref{trace-d} below shows that this is a trace functional
\begin{equation}
\Tr_{\ad}( a* b- b*a)=0,\ \forall\ a,b\in
\Psi^{-\infty}_{\sus(2)}(Z;E)[\epsilon/\epsilon^{2}].
\label{td.2}\end{equation}
Consider the $1$-form
\begin{equation}
\alpha(a)= \Tr_{\ad}(a^{-1}* da)\text{ on }G^{-\infty}(Z;E)[\epsilon/\epsilon^{2}],
\label{td.3}\end{equation}
where the inverse of $a\in
G^{-\infty}_{\sus(2)}(Z;E)[\epsilon/\epsilon^{2}]$ is with respect to the
the truncated $*$-product
\begin{equation}
a^{-1}=a_0^{-1}-\epsilon a_0^{-1}\left(a_1+
\frac{1}{2i}\left(\frac{\pa a_0}{\pa t}a_{0}^{-1}\frac{\pa a_0}{\pa\tau}
-\frac{\pa a_0}{\pa\tau}a_{0}^{-1}\frac{\pa a_0}{\pa t}\right)
\right)a_0^{-1}.
\label{efatopi.124}\end{equation}
The adiabatic determinant is then defined by
\begin{equation}
{\det}_{\ad}(g)=\exp \left( \int_{[0,1]} \gamma^{*} \alpha \right)
\label{td.4}\end{equation}
where $\gamma; [0,1]\longrightarrow
G^{-\infty}_{\sus(2)}(Z;E)[\epsilon/\epsilon^{2}]$ is any smooth path from
the identity to $g$.  Since the integral of 
$\alpha$ along a loop gives an integer multiple of $2\pi i$ (see for
instance proposition 4.4 in \cite{fipomb2}), this definition does not
depend on the choice of $\gamma$.  From \eqref{td.2},
\begin{equation}
\alpha(ab)=\Tr_{\ad}((a * b)^{-1} d(a * b))=
\Tr_{\ad}(a^{-1} * da) + \Tr_{\ad}( b^{-1} *db),
\label{td.5}\end{equation}
and hence
\begin{equation}
m^*\alpha =\pi_L^*\alpha +\pi_R^*\alpha
\label{efatopi.133}\end{equation}
where 
\begin{multline}
m:G^{-\infty}_{\sus(2)}(Z;E)[\epsilon /\epsilon^2]\times
G^{-\infty}_{\sus(2)}(Z;E)[\epsilon /\epsilon^2]\ni (a,b)\longmapsto\\
a[*]b\in G^{-\infty}_{\sus(2)}(Z;E)[\epsilon /\epsilon^2].
\label{efatopi.122}\end{multline}
is the composition given by the truncated $*$-product while $\pi_{L}$ and
$\pi_{R}$ are the projections on the left and right factor.  The
multiplicativity of the adiabatic determinant follows directly from
\eqref{efatopi.133}.

\section{The determinant line bundle}\label{Det-line}

We next describe the construction, and especially primitivity, of the
determinant line bundle over a smooth classifying group for even K-theory.

\begin{definition}\label{efatopi.78} A \emph{primitive line bundle} over a
  (Fr\'echet-Lie) group  
\begin{equation}
\xymatrix{
\cL\ar[d]\\
\cG
}
\label{efatopi.79}\end{equation}
is a smooth, and locally trivial, line bundle equipped with an isomorphism
of the line bundles
\begin{equation}
\begin{gathered}
\pi_L^*\cL\otimes\pi_R^*\cL\overset{\simeq}
\longrightarrow m^*\cL\Mover \cG\times\cG,\\
\pi_L:\cG\times\cG\ni(a,b)\longrightarrow a\in\cG,\
\pi_R:\cG\times\cG\ni(a,b)\longrightarrow b\in\cG,\\
m:\cG\times\cG\ni(a,b)\longrightarrow ab\in\cG
\end{gathered}
\label{efatopi.80}\end{equation}
which is associative in the sense that for any three elements, $a,$ $b,$
$c\in\cG,$ the two induced isomorphisms
\begin{equation}
\xymatrix{
&\cL_{ab}\otimes\cL_c\ar[dr]\\
\cL_a\otimes\cL_b\otimes\cL_c\ar[ur]\ar[dr]\ar@{-->}[rr]&&\cL_{abc}\\
&\cL_{a}\otimes\cL_{bc}\ar[ur]\\
}
\label{efatopi.68}\end{equation}
are the same.
\end{definition}

For the \emph{reduced} classifying group, $G^{-\infty}_{\sus,\ind=0}(Z;E),$
a construction of the determinant line bundle, with this primitivity
property, was given in \cite{arXiv:math/0606382}, although only in the `geometric case'. A variant of the construction there,
also depending heavily on the properties of the suspended determinant but
using instead the `dressed' delooping sequence (for the loop group)
\begin{equation}
\xymatrix{
G^{-\infty}_{\sus(2)}(Z;E)[\epsilon /\epsilon^2]\ar[r]&
\tilde G^{-\infty}_{\sus(2)}(Z;E)[\epsilon /\epsilon^2]\ar[r]&
G^{-\infty}_{\sus,\ind=0}(Z;E)
}
\label{efatopi.117}\end{equation}
again constructs the determinant line bundle, with primitivity condition, over the
component of the identity in the loop group. In this section, by modifying
an idea from the book of Pressley and Segal, \cite{Pressley-Segal1}, we
show how to extend this primitive line bundle to the whole of the
classifying group.

In \eqref{efatopi.117} the central, contractible, group is based on the
half-open but smooth-flat loop group:
\begin{multline}
\tilde G^{-\infty}_{\sus(2)}(Z;E)=\bigg\{
\tilde a:\bbR^2_{(t,\tau)}\longrightarrow\Psi^{-\infty}(Z;E);
\lim_{t\to-\infty}\tilde a(t,\tau)=0,\\
\frac{\pa\tilde a}{\pa t}\in\cS(\bbR^2;\Psi^{-\infty}(Z;E)),\
\tilde a(t,\tau),\ \tilde a(\infty,\tau)\in G^{-\infty}(Z;E)\
\forall\ t,\tau\in\bbR\bigg\}.
\label{efatopi.55}\end{multline}
Note that automatically, $\lim_{\tau\to\infty}\tilde a(t,\tau)=0$ for all
$t\in[-\infty,\infty].$ This group has an extension with the product having
the same `correction term' given by the Poisson bracket on $\bbR^2$ as in
\eqref{efatopi.53}:
\begin{equation}
\tilde G^{-\infty}_{\sus(2)}(Z;E)[\epsilon /\epsilon
  ^2]=\tilde G^{-\infty}_{\sus(2)}(Z;E)\oplus\Psi^{-\infty}_{\sus(2)}(Z;E)
\label{efatopi.56}\end{equation}
where the additional terms at level $\epsilon$ are just Schwartz functions
valued in the smoothing operators without any additional
invertibility. Note that the term in the product involving the Poisson
bracket always leads to a Schwartz function on $\bbR^2,$ since one factor
is differentiated with respect to $t.$ Thus $\tilde
G^{-\infty}_{\sus(2)}(Z;E)[\epsilon /\epsilon ^2]$ is again contractible,
with just the addition of a lower order `affine' term.

To expand the quotient group to the whole classifying group, choose one
element $s\in G^{-\infty}_{\sus}(Z;E)$ of index $1.$ Then $s^j$ is in the
component of index $j$ so each element $a\in G^{-\infty}_{\sus}(Z;E)$ can
be connected by a curve, and hence a flat-smooth loop, to $s^j$ for
precisely one $j.$ The group in \eqref{efatopi.55} may then be further
enlarged to
\begin{multline}
\cD^{-\infty}_{\sus(2)}(Z;E)=\bigg\{
\tilde a:\bbR^2_{(t,\tau)}\longrightarrow\Psi^{-\infty}(Z;E);
\frac{\pa\tilde a}{\pa t}\in\cS(\bbR^2;\Psi^{-\infty}(Z;E)),\\
\lim_{t\to-\infty}\tilde a(t,\tau)=s^j\Mforsome j,\
\tilde a(t,\tau),\ \tilde a(\infty,\tau)\in G^{-\infty}(Z;E)\
\forall\ t,\tau\in\bbR\bigg\}.
\label{efatopi.69}\end{multline}
This expanded group has countably many components, labelled by $j,$ and
the restriction map to $t=\infty$ is a surjection to
$G^{-\infty}_{\sus}(Z;E).$ Thus, after adding the same affine lower order
terms, \eqref{efatopi.117} is replaced by the new short exact sequence
\begin{equation}
\xymatrix{
G^{-\infty}_{\sus(2)}(Z;E)[\epsilon /\epsilon^2]\ar@{^(->}[r]&
\cD^{-\infty}_{\sus(2)}(Z;E)[\epsilon /\epsilon^2]\ar[r]^-{\tilde R_{\infty}}&
G^{-\infty}_{\sus}(Z;E).
}
\label{efatopi.77}\end{equation}
The central group is no longer contractible, although each of its connected
component is. However the 1-form $\alpha$ in \eqref{td.3} can be extended
to give a smooth 1-form on $\cD^{-\infty}_{\sus(2)}(Z;E)[\epsilon /\epsilon^2].$
Indeed, the adiabatic trace has an obvious extension to a functional on
\begin{equation}
\widetilde{\Psi}^{-\infty}_{\sus(2)}(Z;E)[\epsilon/\epsilon^{2}]=
\widetilde{\Psi}^{-\infty}_{\sus(2)}(Z;E) \oplus\Psi^{-\infty}_{\sus(2)}(Z;E)
\label{td.6}\end{equation}
where 
\begin{multline}
  \widetilde{\Psi}^{-\infty}_{\sus(2)}(Z;E)= \left\{ a\in \CI(\bbR^{2};\Psi^{-\infty}(Z;E);
           \frac{\pa a}{\pa t}\in \cS(\bbR^{2};\Psi^{-\infty}(Z;E)), \right. \\
           \left.  \lim_{t\to -\infty} a(t,\tau)=0 \right\}.
\label{td.8}\end{multline}
Namely
\begin{equation}
\widetilde{\Tr}_{\ad}(a)= \frac{1}{2\pi} \int_{\bbR^{2}} \Tr_{Z}(a_{1}) dt d\tau, 
\ a=a_{0}+\epsilon a_{1} \in
\widetilde{\Psi}^{-\infty}_{\sus(2)}(Z;E)[\epsilon/\epsilon^{2}]
\label{td.7}\end{equation}
Thus, on $\cD^{-\infty}_{\sus(2)}(Z;E)[\epsilon/\epsilon^{2}]$, one can
consider the smooth $1$-form
\begin{equation}
 \widetilde{\alpha}(a)=
\frac{1}{2} \widetilde{\Tr}_{\ad} ( a^{-1}* da + da * a^{-1})
\label{efatopi.134}\end{equation}
which restricts to $\alpha$ on
$G^{-\infty}_{\sus(2)}(Z;E)[\epsilon/\epsilon^{2}].$
\begin{lemma}
For $a=a_{0}+\epsilon a_{1}$ and $b=b_{0}+\epsilon b_{1}$ in $\widetilde{\Psi}^{-\infty}_{\sus(2)}(Z;E)[\epsilon/\epsilon^{2}],$
\begin{equation*}
\widetilde{\Tr}_{\ad}( a*b-b*a)= \frac{1}{2\pi i} \int_{\bbR} \Tr_{Z}\left( 
 \frac{\pa a_{0}}{\pa \tau}(\infty,\tau) b_{0}(\infty,\tau) \right) d\tau.
\label{efatopi.183}\end{equation*}
In particular, this trace-defect vanishes if $a,b \in
\Psi^{-\infty}_{\sus(2)}(Z;E)[\epsilon/\epsilon^{2}].$  
\label{trace-d}\end{lemma} 
\begin{proof}
By definition of the truncated $*$-product and using the trace property of
$\Tr_{Z},$
\begin{equation}
\widetilde{\Tr}_{\ad}(a* b- b*a)= -\frac{1}{2\pi i} \int_{\bbR^{2}}\Tr_{Z}
\left(\frac{\pa a_{0}}{\pa t} \frac{\pa b_{0}}{\pa \tau}-
\frac{\pa a_{0}}{\pa \tau}\frac{\pa b_{0}}{\pa t}  \right)dt d\tau.
\label{td.9}\end{equation} 
Integrating by parts the first term on the right,
\begin{equation}
\begin{aligned}
\int_{\bbR^{2}} \Tr_{Z} \left( \frac{\pa a_{0}}{\pa t}
\frac{\pa b_{0}}{\pa \tau}\right)& dt d\tau  =
  -\int_{\bbR^{2}} \Tr_{Z} \left(  \frac{\pa^{2} a_{0}}{\pa \tau \pa t} b_{0} \right) dt d\tau \\
&= -\int_{\bbR^{2}} \frac{\pa}{\pa t}
\left(  \Tr_{Z} \left( \frac{\pa a_{0}}{\pa \tau} b_{0}\right) \right)dt d\tau
+\int_{\bbR^{2}} \Tr_{Z}\left(\frac{\pa a_{0}}{\pa \tau}\frac{\pa b_{0}}{\pa t}
\right) dt d\tau.
\end{aligned}
\label{td.10}\end{equation}
Thus,
\begin{equation}
\begin{aligned}
\Tr_{\ad}( a*b- b*a) & = \frac{1}{2\pi i} \int_{\bbR^{2}}
\frac{\pa}{\pa t} \left(  \Tr_{Z} \left( \frac{\pa a_{0}}{\pa \tau}
b_{0}\right) \right) dt d\tau, \\ 
  &= \frac{1}{2\pi i} \int_{\bbR} \Tr_{Z} \left(  \frac{\pa a_{0}}{\pa
  \tau}(\infty,\tau) b_{0}(\infty,\tau) \right) d\tau.
\end{aligned}
\label{td.11}\end{equation}
\end{proof}

\begin{proposition}\label{efatopi.126} Under the maps on the product 
\begin{equation}
\xymatrix@C=-2pc{
&\cD^{-\infty}_{\sus(2)}(Z;E)[\epsilon /\epsilon^2]&\\
&
\cD^{-\infty}_{\sus(2)}(Z;E)[\epsilon /\epsilon^2]\times
\cD^{-\infty}_{\sus(2)}(Z;E)[\epsilon /\epsilon^2]
\ar[u]^m\ar[dl]^{\pi_L}\ar[dr]_{\pi_R}
\ar[dd]^{\tilde R_{\infty}\times \tilde R_{\infty}}
\\
\cD^{-\infty}_{\sus(2)}(Z;E)[\epsilon /\epsilon^2]&&
\cD^{-\infty}_{\sus(2)}(Z;E)[\epsilon /\epsilon^2]\\
&
(G^{-\infty}_{\sus}(Z,E))^2
}
\label{efatopi.127}\end{equation}
the $1$-form $\tilde\alpha$ in \eqref{efatopi.134} satisfies 
\begin{equation}
m^*\tilde \alpha =\pi_L^*\tilde\alpha +\pi_R^*\tilde\alpha+(\tilde
R_{\infty}\times\tilde R_{\infty})^*\delta,
\label{efatopi.136}\end{equation}
with  
\begin{equation*}
\delta(a,b)=- \frac{1}{4\pi i} \int_{\bbR}\Tr_Z\left(a^{-1}(da)\frac{\pa b}{\pa\tau}b^{-1}
-(db)b^{-1}a^{-1}\frac{\pa a}{\pa\tau}\right)d\tau
\label{efatopi.184}\end{equation*}
on $G^{-\infty}_{\sus}(Z;E)\times G^{-\infty}_{\sus}(Z;E).$
\end{proposition}
\begin{proof} 
If $a,b\in \cD^{-\infty}_{\sus(2)}(Z;E)[\epsilon/\epsilon^{2}],$ the
trace-defect formula of Lemma~\ref{trace-d} gives
\begin{equation}
\begin{aligned}
\widetilde{\Tr}_{\ad}((a*b)^{-1}*&d(a*b)) =
\widetilde{\Tr}_{\ad} ( b^{-1}*a^{-1}*da*b+ b^{-1}*db ) \\
&= \widetilde{\Tr}_{\ad}( a^{-1}*da)+ \widetilde{\Tr}_{\ad}( b^{-1}*db)  \\
& + \widetilde{\Tr}_{\ad}( b^{-1}*(a^{-1}*da*b)- (a^{-1}*da*b)*b^{-1} )  \\
&= \widetilde{\Tr}_{\ad}( a^{-1}*da)+ \widetilde{\Tr}_{\ad}( b^{-1}*db)  \\
& +\frac{1}{2\pi i}\int_{\bbR}\Tr_{Z}
\left( \frac{\pa b_{0}^{-1}}{\pa \tau}(\infty,\tau) 
(a_{0}^{-1}da_{0} b_{0})(\infty,\tau) \right) d\tau  \\
&=\widetilde{\Tr}_{\ad}( a^{-1}*da)+ \widetilde{\Tr}_{\ad}( b^{-1}*db)  \\
&-\frac{1}{2\pi i}\int_{\bbR}
\Tr_{Z}\left( \frac{\pa b_{0}}{\pa \tau}(\infty,\tau) b_{0}^{-1}(\infty,\tau)
 a_{0}^{-1}(\infty,\tau) da_{0}(\infty,\tau) \right) d\tau.
\end{aligned}
\label{td.12}\end{equation}
Similarly,
\begin{multline}
\widetilde{\Tr}_{\ad}(d(a*b)*(a*b)^{-1})= \widetilde{\Tr}_{\ad}( da*a^{-1}) + \widetilde{\Tr}_{\ad}( 
db* b^{-1})\\
+\frac{1}{2\pi i} \int_{\bbR}
\Tr_{Z} \left( db_{0}(\infty,\tau) b^{-1}(\infty,\tau) a^{-1}(\infty,\tau) 
\frac{\pa a_{0}}{\pa \tau}(\infty,\tau) \right) d\tau.
\label{td.13}\end{multline}
Combining these two computations, the result follows.
\end{proof}

\begin{proposition}\label{efatopi.72} The adiabatic determinant on the
  normal subgroup in \eqref{efatopi.77} induces the determinant line
  bundle, $\cL,$ which is primitive over the quotient and $\tilde\alpha$ in
  \eqref{efatopi.134} defines a connection $\nabla_{\ad}$ on $\cL$ with
  curvature form the $2$-form part of the universal even Chern character
  of \eqref{efatopi.65}.
\end{proposition}

\begin{proof} The form $\tilde\alpha$ in \eqref{efatopi.134} restricts to
 $\alpha$ in \eqref{td.3} on $G^{-\infty}_{\sus(2)}(Z;E)[\epsilon
  /\epsilon^2].$ The latter is the differential of the logarithm of
  $\det_{\ad}.$ As a special case of Proposition~\ref{efatopi.126} above,
the first factor may be restricted to $G^{-\infty}_{\sus(2)}(Z;E),$
and then $\delta$ in \eqref{efatopi.136} vanishes since $a\equiv0.$ This
shows that as a connection on the trivial bundle over
  $\cD^{-\infty}_{\sus(2)}(Z;E)[\epsilon /\epsilon^2],$ $d-\tilde\alpha$ is
  invariant under the left action of $G^{-\infty}_{\sus(2)}(Z;E)[\epsilon
    /\epsilon^2],$ acting through the adiabatic determinant on the
  fibres. Thus $d-\tilde\alpha$ projects to a connection $\nabla_{\ad}$ on
  the determinant line bundle over $G^{-\infty}_{\sus}(Z;E)$ defined as the
  quotient by this action, i\@.e\@.~as the line bundle induced by
  $\det_{\ad}$ as a representation of the structure group.

To compute the curvature we simply need to compute the differential of
$\tilde \alpha.$  Using the trace-defect formula of Lemma~\ref{trace-d},
\begin{equation}
\begin{aligned}
d \widetilde{\Tr}_{\ad}(a^{-1}*da) &= -\widetilde{\Tr}_{\ad}(a^{-1}*da*a^{-1}*da)= -\frac{1}{2} \widetilde{\Tr}_{\ad}( [a^{-1}*da, a^{-1}*da]), \\
&= -\frac{1}{4\pi i} \int_{\bbR} \Tr_{Z} \left(   
       \left(\frac{\pa}{\pa \tau} (\sigma^{-1}d\sigma)\right) \sigma^{-1}d\sigma    \right) d\tau, \; \mbox{with} \; \sigma=a_{0}(\infty,\tau), \\
&= \frac{1}{4\pi i} \int_{\bbR} \Tr_{Z} \left(
   \sigma^{-1}\frac{\pa \sigma}{\pa \tau}(\sigma^{-1}d\sigma)^{2} 
    - \sigma^{-1}d\left( \frac{\pa \sigma}{\pa \tau}\right)
 \sigma^{-1}d\sigma    
  \right)d\tau.  
\end{aligned}
\label{curv.1}\end{equation}
Similarly, we compute that
\begin{equation}
\begin{aligned}
d\widetilde{\Tr}_{\ad}(da*a^{-1}) &=
\widetilde{\Tr}_{\ad}(da*a^{-1}*da*a^{-1})=
\frac{1}{2} \widetilde{\Tr}_{\ad}( [da*a^{-1}, da*a^{-1}]), \\
&= \frac{1}{4\pi i} \int_{\bbR} \Tr_{Z} \left(   
\left(\frac{\pa}{\pa \tau} (d\sigma\sigma^{-1})\right) d\sigma\sigma^{-1}
\right) d\tau, \\
&= -\frac{1}{4\pi i} \int_{\bbR} \Tr_{Z} \left(
   d\sigma\sigma^{-1}\frac{\pa \sigma}{\pa \tau}\sigma^{-1}d\sigma \sigma^{-1}
    - d\left( \frac{\pa \sigma}{\pa \tau}\right)
 \sigma^{-1}d\sigma \sigma^{-1}    
  \right)d\tau.
\end{aligned}
\label{curv.3}\end{equation}
Recall that the $2$-form part of the universal even Chern character on
$G^{-\infty}_{\sus}(Z;E)$ is given by (\cf formula (3.7) in \cite{fipomb})
\begin{equation}
  (\Ch_{\even})_{[2]}= \frac{1}{2(2\pi i)^{2}} \int_{\bbR}
\Tr_{Z}\left( \sigma^{-1}\frac{\pa \sigma}{\pa \tau}(\sigma^{-1}d\sigma)^{2}
\right) d\tau.
\label{curv.4}\end{equation}
Thus, combining \eqref{curv.1} and \eqref{curv.3},
\begin{equation}
d\tilde{\alpha}= \frac{1}{4\pi i} \int_{\bbR} \Tr_{Z} \left(
   \sigma^{-1}\frac{\pa \sigma}{\pa \tau}(\sigma^{-1}d\sigma)^{2}     
  \right)d\tau = 2\pi i \tilde{R}_{\infty}^{*}( (\Ch_{\even})_{[2]}),
\label{curv.5}\end{equation}
that is,
\begin{equation}
\frac{i}{2\pi} \nabla_{\ad}^{2}=\tilde{R}_{\infty}^{*}( (\Ch_{\even})_{[2]}). 
\label{curv.6}\end{equation}
\end{proof}

Next, this construction of the determinant bundle is extended to the
geometric case. The sequence, \eqref{efatopi.77}, being natural, extends to
give smooth bundles over the fibres of \eqref{efatopi.1}: 
\begin{equation}
\xymatrix{
G^{-\infty}_{\sus(2)}(\phi;E)[\epsilon /\epsilon^2]\ar@{^(->}[r]&
\cD^{-\infty}_{\sus(2)}(\phi;E)[\epsilon /\epsilon^2]\ar[r]^-{\tilde R_{\infty}}&
G^{-\infty}_{\sus}(\phi;E).
}
\label{efatopi.180}\end{equation}
Furthermore, using the connection chosen earlier, the form $\tilde\alpha$
in \eqref{efatopi.134} can be replaced by $\tilde\alpha _{\phi}$ by
substituting $\nabla^{\phi,E}$ for $d$ throughout. The resulting $1$-form is
well-defined on $\cD^{-\infty}_{\sus(2)}(\phi;E).$ Moreover the
proof of proposition~\ref{efatopi.126} only depends on the derivation property of $d$ so
extend directly to $\tilde\alpha_{\phi}.$ In particular \eqref{efatopi.136}
carries over to the fibre products. This leads to the following
geometric version of Proposition~\ref{efatopi.72}.

\begin{proposition}\label{efatopi.182} The adiabatic determinant on the
fibres of the structure bundle in \eqref{efatopi.77} induces the determinant line
  bundle, $\cL,$ over $G^{-\infty}_{\sus}(\phi;E);$ the 1-form
  $\tilde\alpha_{\phi}$ defines a connection
  $\nabla_{\phi}$ on $\cL$ with curvature the $2$-form part of the even
  Chern character on $G^{-\infty}_{\sus}(\phi;E).$
\end{proposition}
\begin{proof}
What is slightly different in the geometric case is the computation of the
curvature of $\nabla_{\phi}$.  Taking into account \eqref{gcf.5a}, the
analogue of \eqref{curv.1} and \eqref{curv.3} is 
\begin{equation}
 d\tilde\alpha_{\phi}= \frac{1}{4\pi i} \int_{\bbR} 
  \Tr_{Z}\left(\sigma^{-1}\frac{\pa \sigma}{\pa \tau} (\sigma^{-1}\nabla^{\phi,E}\sigma)^{2}   \right) d\tau + \frac{1}{2}\Tr_{\ad}\left( 
a^{-1}*([\omega,a])+([\omega,a])*a^{-1} \right)
\label{gcurv.1}\end{equation}
where $\sigma=a_{0}(\infty,\tau)$.  To compute the second term, we use the 
identity $a^{-1}*(a\omega)= (\omega a)*a^{-1}=\omega$ to rewrite it as
\begin{multline}
\frac{1}{2}\Tr_{\ad}\left( 
a^{-1}*([\omega,a])+([\omega,a])*a^{-1} \right)=  \\
\frac{1}{2}\Tr_{\ad}\left( a^{-1}*(\omega a)- (\omega a)*a^{-1}\right)
+\frac{1}{2} \Tr_{\ad}\left( a^{-1}*(a\omega)- (a\omega)*a^{-1}\right).
\label{gcurv.2}\end{multline}
Using the trace-defect formula of lemma~\ref{trace-d}, this gives
\begin{equation}
\begin{aligned}
d\tilde\alpha_{\phi}&= \frac{1}{4\pi i} \int_{\bbR} 
\Tr_{Z}\left( \sigma^{-1}\frac{\pa \sigma}{\pa \tau} \left(
   (\sigma^{-1}\nabla^{\phi,E}\sigma)^{2}- \sigma^{-1}\omega\sigma-\omega\right)   \right) d\tau \\
 &= 2\pi i \tilde{R}_{\infty}^{*}( (\Ch_{\even}(\nabla^{\phi,E}))_{[2]}),
\end{aligned}
\label{gcurv.3}\end{equation}
from which the result follows.
\end{proof}

\section{The K-theory gerbe}\label{K-gerbe}

First we consider the universal K-theory gerbe, i\@.e\@.~the gerbe over the
classifying space $G^{-\infty}(Z;E)$ for odd K-theory. Such a gerbe was
originally introduced by Carey and Mickelsson \cite{Carey-Mickelsson1}, 
\cite{Carey-Mickelsson2} over a slightly different classifying space for
odd K-theory, namely the space of unitary operators which are perturbations
of the identity by operators of trace class.
We propose a different construction of the universal K-theory gerbe using the
determinant line bundle of proposition~\ref{efatopi.72}. 

Recall that the
delooping sequence \eqref{efatopi.6} for a single manifold 
\begin{equation}
\xymatrix{
\cL\ar[d]^{\pi}\\
G^{-\infty}_{\sus}(Z;E)\ar[r]&
\tilde G^{-\infty}_{\sus}(Z;E)\ar[d]^{R_{\infty}}\\
&G^{-\infty}(Z;E)}
\label{efatopi.137}\end{equation}
is a classifying sequence for K-theory, the normal subgroup is classifying
for even K-theory, the central group is contractible and the quotient is
classifying for odd K-theory. Moreover, in the preceding section, we have
constructed the smooth primitive determinant line bundle over
$G^{-\infty}_{\sus}(Z;E)$ with connection $\nabla_{\ad}$ given in
Proposition~\ref{efatopi.72}. This induces the K-theory gerbe over the
classifying space, as a line bundle over the fibre product of two copies of
the fibration $R_{\infty}:$ 
\begin{equation}
\xymatrix{
&\tilde m^*\cL\ar[d]^{\pi}
&
\cL\ar[d]^{\pi}
\\
\tilde G^{-\infty}_{\sus}(Z;E)\ar[dr]_{R_{\infty}}
&
(\tilde G^{-\infty}_{\sus}(Z;E))^{[2]}\ar[d]^{(R_{\infty})^{[2]}}
\ar[d]\ar@<2pt>[l]^-{\pi_R}\ar@<-2pt>[l]_-{\pi_L}
\ar[r]^-{\tilde m}
&
G^{-\infty}_{\sus}(Z;E)
\\
&G^{-\infty}(Z;E).
}
\label{efatopi.138}\end{equation}
Here $\tilde m:(\tilde G^{-\infty}_{\sus}(Z;E))^{[2]}\longrightarrow
G^{-\infty}_{\sus}(Z;E)$ is the fibre-shift map $\tilde m(a,b)=ab^{-1},$ where
$R_{\infty}(a)=R_{\infty}(b),$ by definition of the fibre product, so
$\tilde m(a,b)\in G^{-\infty}_{\sus}(Z;E)$ by the exactness of
\eqref{efatopi.137}, as indicated.

\begin{theorem}\label{efatopi.139} There is a
  connection $\tilde\nabla_{\ad}$ on $\tilde m^*\cL$ with curvature
\begin{equation}
\frac{i}{2\pi}F_{\tilde\nabla_{\ad}}=\pi^*_L\tilde\eta_2-\pi^*_R\tilde\eta_2
\Mon (\tilde G^{-\infty}_{\sus}(Z;E))^{[2]}
\label{efatopi.150}\end{equation}
where the B-field, $\tilde\eta_2,$ is the 2-form part of the eta form in
\eqref{efatopi.20} which has basic differential the 3-form part of the odd
Chern character on $G^{-\infty}(Z;E),$ as shown by \eqref{efatopi.22}.
\end{theorem}

\begin{proof} The connection $\nabla_{\ad}$ on $\cL$ as a bundle over
  $G^{-\infty}_{\sus}(Z;E)$ given by Proposition~\ref{efatopi.72} pulls
  back to a connection $\tilde m^*\nabla_{\ad}$ on $\tilde m^*\cL.$ The
  curvature is just the pull-back of the curvature on
  $G^{-\infty}_{\sus}(Z;E)$ and again by Proposition~\ref{efatopi.72} this
  is the 2-form part of the Chern character. By
  Proposition~\ref{efatopi.21}, the 2-form part of the eta form on $\tilde
  G^{-\infty}_{\sus}(Z;E)$ pulls back under the product map as in
  \eqref{efatopi.149}. To apply this result here we need to invert the
  right factor, to change from $m$ to $\tilde m,$ which also has the effect of
  changing the sign of the eta form from that factor leading to
\begin{equation}
\tilde m^*\tilde\eta=\pi_L^*\tilde\eta-\pi_R^*\tilde\eta+d(\tilde\delta' _{\odd})
+(R_{\infty}\times R_{\infty})^*\delta' _{\even}
\label{efatopi.152}\end{equation}
where the primes indicate that the forms are first pulled back under
inversion in the second variable. Now, restricting \eqref{efatopi.152} to
the fibre diagonal gives
\begin{equation}
\tilde m^*\Ch_{\even}=\pi_L^*\tilde\eta-\pi_R^*\tilde\eta+d(\tilde\delta' _{\odd})
\label{efatopi.151}\end{equation}
since the last term now factors through the constant map to the identity.
This corresponds to the middle row of \eqref{efatopi.138}. In
particular, if the connection is modified by the 1-form part of $\delta
'_{\odd}$
\begin{equation}
\tilde\nabla_{\ad}=\tilde m^*\nabla_{\ad}+2\pi i(\tilde\delta '_{\odd})_1
\label{efatopi.153}\end{equation}
then it has curvature as claimed 
\begin{equation}
\frac{i}{2\pi}(\tilde\nabla_{\ad})^2=\tilde m^*(\Ch_{\even})_2- d(\delta'_{\odd})_{1}=
\pi_L^*\tilde\eta_2-\pi_R^*\tilde\eta_2 
\label{efatopi.154}\end{equation}
which is precisely the statement that $\tilde\eta_2$ is a B-field for the
gerbe. The curving of the gerbe is then the basic form of which the
differential of the B-field is the pull-back and from \eqref{efatopi.22} 
\begin{equation}
d\tilde\eta_2=R_{\infty}^*(\Ch_{\odd})_3.
\label{efatopi.155}\end{equation}
\end{proof}

The bundle gerbe \eqref{efatopi.138} with connection given by 
theorem~\ref{efatopi.139} is universal in the sense that given an
odd K-theory class $[g]\in K^{1}(Y)$ represented by a smooth map
$g:Y\to G^{-\infty}(Z;E)$, the pull-back of \eqref{efatopi.138} to
$Y$ by $g$ gives a bundle gerbe (with connection) on $Y$ whose
curving 3-form is given by $g^{*}\Ch_{\odd}$ (\cf Theorem 5.1 in
\cite{Carey-Wang2006}).

Since 
\begin{equation}
    (\Ch_{\odd})_{3}= \frac{1}{6(2\pi i)^{2}} \Tr( (\sigma^{-1}d\sigma)^{3})
\label{basic.1}\end{equation}
is the image in $H^{3}(G^{-\infty}(Z;E);\bbC)$ of the generator of
$H^{3}(G^{-\infty}(Z;E);\bbZ)\cong \bbZ$, we also note that the 
bundle gerbe \eqref{efatopi.138} is basic for the group 
$G^{-\infty}(Z;E)$.  By considering an $n$-dimensional subspace of
$L^{2}(Z;E)$ with norm defined by a choice of metric on $Z$ and of Hermitian
metric on $E$, we get a natural inclusion $\U(n)\subset G^{-\infty}(Z;E)$.  Pulling back \eqref{efatopi.138} to $\U(n)$ via this inclusion gives the
basic bundle gerbe $\U(n)$.  This is a `smooth' construction in any
reasonable sense although it is infinite dimensional in nature.  

Infinite dimensional constructions of the basic bundle gerbe of a Lie group
first appeared in the book of Brylinski \cite{Brylinski} and later
in \cite{Brylinski-McLaughlin}.  The tautological bundle gerbe of Murray
\cite{Murray1} for $2$-connected manifolds also provides such a construction
for simply connected Lie groups (see also \cite{Carey-Murray-Wang}).  More
recently, finite dimensional constructions of the basic gerbe were
obtained by Gawedzki and Reis \cite{Gawedzki-Reis2002} for $\SU(n)$
and shortly after by Meinrenken \cite{Meinrenken2003} for simple simply 
connected Lie groups.  The construction of Meinrenken was subsequently
generalized to non-simply connected Lie groups in \cite{Gawedzki-Reis2004}
and \cite{Murray-Stevenson2008}.

\section{Geometric Gerbe for an odd elliptic family}\label{Ell-gerbe}

As pointed out in \cite{Carey-Mickelsson-Murray1997}, gerbes are intimately
related to index theory.  In our case, we have the following construction
of the index gerbe (with connection) associated to a family of self-adjoint elliptic pseudodifferential operators and more generally to  a product-type family
of fully elliptic operators.

\begin{theorem}\label{efatopi.60} Let $A\in\Psi^{1}(M/Y;E)$ be a self-adjoint
  elliptic family as in Section~\ref{EtaFam}, or a product-type fully elliptic 
  family $A(t)\in\Psi^{m,l}_{\psus}(M/Y;E)$ then the determinant bundle
  induces a bundle-gerbe 
\begin{equation}
\xymatrix{
&S^*\cL\ar[d]^{\pi}
&
\cL\ar[d]^{\pi}
\\
\cA\ar[dr]_{p}
&
\cA^{[2]}\ar[d]^{p^{[2]}}
\ar[d]\ar@<2pt>[l]^-{\pi_R}\ar@<-2pt>[l]_-{\pi_L}
\ar[r]^-S
&
G^{-\infty}_{\sus}(\phi;E)
\\
&Y.
}
\label{efatopi.61}\end{equation}
The connection on $S^*\cL$ 
\begin{equation}
\nabla_{\cA}=S^*\nabla+\gamma ,\ \gamma=2\pi i(\tilde\delta_{\cA})_1,
\label{efatopi.62}\end{equation}
given by the $1$-form part of the form in \eqref{efatopi.178},
is primitive in the sense that the curvature on $\cA^{[2]}$ splits 
\begin{equation}
\frac{i}{2\pi}(\nabla_{\cA})^2=\pi_L^*\eta _{\cA,2}-\pi_R^*\eta_{\cA,2}
\label{efatopi.63}\end{equation}
showing that $\eta_{\cA,2}$ is a B-field and that the gerbe has curving
3-form  
\begin{equation}
d\eta_{\cA,2}=p^*\Ch_{\cA,3}
\label{efatopi.64}\end{equation}
the 3-form part of the Chern character of the index bundle of the family.
\end{theorem}

\begin{proof} This result follows by an argument parallel to the preceding
  one, given \eqref{efatopi.178}, Proposition~\ref{efatopi.72},
  Proposition~\ref{efatopi.182} and the 3-form part of \eqref{efatopi.113}.
\end{proof}

This result can be seen as a pseudodifferential generalization of a 
result of Lott (Theorem 1 in \cite{Lott2003}).  See also 
\cite{Bunke2002} for a different treatment of the index gerbe and
\cite{Carey-Wang2006} for a generalization of (\cite{Lott2003}, Theorem 1)
to families of Dirac operators on odd dimensional manifolds with boudary.

\section{Relation with the Bismut-Cheeger eta form} \label{bs.0}

Amongst the most important geometric examples of self-adjoint elliptic
operators are the Dirac-type operators on odd dimensional
manifolds. Suppose now that the fibres of the fibration \eqref{efatopi.1}
are odd dimensional.  Let $g\in\CI(M; \odot^{2} T^{*}(M/Y))$ be a family of
fibrewise metrics and let $\cli(T(M/Y))$ be the associated bundle of Clifford
algebras for the vertical tangent bundle $T(M/Y).$ Let $\bbE\to M$ be a
Clifford module with respect to $\cli(T(M/Y))$ with Clifford action $c:
\cli(T(M/Y))\to \End(\bbE).$ Let also $\nabla^{\bbE}$ be a family of
fibrewise Clifford connections, that is a family of unitary connections such that
\[
[\nabla^{\bbE}_{X_1}, c(X_{2})]=
c(\nabla^{LC}_{X_{1}}X_{2})\in \CI(M;\End(\bbE)),\ \forall\ X_{1},X_{2}\in \CI(M;T(M/Y)),
\]
where $\nabla^{LC}$ is the fibrewise Levi-Civita connection associated to
the family of metrics $g.$ This data allows us to define a family of
Dirac-type operators by
\begin{equation}
      \eth= c\circ \nabla^{\bbE}.
\label{bs.1}\end{equation}
For invertible families of this type, Bismut and Cheeger introduced in
\cite{Bismut-Cheeger} an eta form on the base. Their construction was
subsequently generalized in \cite{MR99a:58144} to situations where the
family is not invertible, but admits a perturbation by a family $Q\in
\Psi^{-\infty}(M/Y;\bbE)$ of self-ajoint smoothing operators such that
$\eth+Q$ is invertible. The odd families index of the family $\eth$ is
precisely the obstruction to the existence of such a family of
perturbations; for the boundary operators of a family of Dirac-type
operators on a fibration of manifolds with boundary, this index obstruction
vanishes by the cobordism invariance of the index, so that invertible
perturbations exist in this case.

When the family $\eth$ is invertible, the Bismut-Cheeger eta form is given by
\begin{equation}
\eta_{\BC}(\eth) =
\frac{1}{\sqrt{\pi}} \int_{0}^{\infty} \STr_{\cli(1)} \left( \frac{d\bbB_{t}}{dt}
      e^{-\bbB^{2}_{t}}\right) dt \; \in \Omega^{\even}(Y),
\label{bs.2}\end{equation}
where $\bbB_{t}$ is the rescaled Bismut superconnection associated to
$\eth$ (see (10.8) and (13.7) in \cite{MR99a:58144} for a detailed
discussion). For a family perturbed to be invertible, $\eth+Q,$ the
definition is slightly modified to
\begin{equation}
\eta_{\BC}(\eth+Q) =
\frac{1}{\sqrt{\pi}} \int_{0}^{\infty}
\STr_{\cli(1)} \left( \frac{d\widetilde{\bbB}_{t}}{dt}
      e^{-\widetilde{\bbB}^{2}_{t}}\right) dt \; \in \Omega^{\even}(Y),
\label{bs.2b}\end{equation}
where $\widetilde{\bbB}_{t}= \bbB_{t}+ t^{\frac{1}{2}}\chi(t) Q\sigma$,
with $\sigma\in\cli(1)$ a generator of $\cli(1)$ such that $\sigma^{2}=1$
and $\chi\in \CI(\bbR)$ is a non-negative function with $\chi(t)=0$ for
$t<1$ and $\chi(t)=1$ for $t>2$.  The Bismut-Cheeger eta form satisfies the following
transgression formula.

\begin{proposition}
The exterior differential of $\eta_{\BC}(\eth+Q)$ does not depend on the
choice of perturbation $Q$ and is given by
\[
 d \eta_{\BC}(\eth+Q)=(2\pi i)^{-\frac{n+1}{2}}
\int_{M/Y} \hat{A}(R_{g})\Ch'(\bbE),
\]
where $n$ is the dimension of the fibres of the fibration $\phi:M\to Y$ and
$\Ch'(\bbE)$ is the twisting Chern character of $\bbE.$   
\label{bs.3}\end{proposition}

\begin{proof}
Because of the cut-off function $\chi$ we have that 
\begin{equation}
\begin{aligned}
     \lim_{t\to 0} \frac{1}{\sqrt{\pi}} \STr_{\cli(1)}\left( e^{-\widetilde{\bbB}^{2}_{t}}\right) & =
     \lim_{t\to 0} \frac{1}{\sqrt{\pi}} \STr_{\cli(1)}\left( e^{-\bbB^{2}_{t}}\right) \\
     &=
      \frac{1}{(2\pi i)^{\frac{n+1}{2}}} \int_{M/Y} \hat{A}(R_{g}) \Ch'(\bbE).
\end{aligned}      
\label{bs.4}\end{equation}
On the other hand, because the family $\eth+Q$ is invertible,
\begin{equation}
  \lim_{t\to \infty} \frac{1}{\sqrt{\pi}} \STr_{\cli(1)}\left( e^{-\widetilde{\bbB}^{2}_{t}}\right)  =0
\label{bs.5}\end{equation}
exponentially fast.  
Thus, the result follows by combining \eqref{bs.4} and \eqref{bs.5} with 
\[
    \frac{d}{dt} \STr_{\cli(1)}\left( e^{-\widetilde{\bbB}^{2}_{t}}\right)=
    -d_{Y}  \STr_{\cli(1)} \left( \frac{d\widetilde{\bbB}_{t}}{dt}
      e^{-\widetilde{\bbB}^{2}_{t}}\right)    \]
and integrating in $t$.      
\end{proof}

If we consider the rescaled version of the Bismut-Cheeger eta form,
\begin{equation}
\widehat{\eta}_{\BC}(\eth+Q) =
\sum_{j=0}^{\infty}(2\pi i)^{-j}\eta_{\BC, [2j]}(\eth+Q),
\label{bs.10}\end{equation}
where $\eta_{\BC, [2j]}$ is the part of $\eta_{\BC}$ of degree $2j,$ then
Proposition~\ref{bs.3} shows that the Chern character of the family index
is trivial in cohomology, which is consistent with the fact the odd index
of $\eth$ must vanish for the invertible perturbation $\eth+Q$ to exist.

It also follows from Proposition~\ref{bs.3} that if $Q_{1}$ and $Q_{2}$ are
two perturbations giving invertible families, then
$\widehat{\eta}_{\BC}(\eth+Q_{1})-\widehat{\eta}_{\BC}(\eth+Q_{2})$ is a
closed form. Moreover, using Proposition~\ref{bs.3} again, it can be seen
that the cohomolgy class represented by the form
$\widehat{\eta}_{\BC}(\eth+Q_{1})-\widehat{\eta}_{\BC}(\eth+Q_{2})$ only
depends on the homotopy classes of $Q_{1}$ and $Q_{2}$ in the space of such
perturbations. This cohomology class can be identified usign the notion of
spectral sections introduced in \cite{MR99a:58144}.

\begin{definition}
A spectral section for the family of self-adjoint Dirac-type operators $\eth$ of
\eqref{bs.1} is a family of self-adjoint projections $P\in
\Psi^{0}(M/Y;\bbE)$ such that for some smooth function $R:Y\to \bbR^{+}$
(depending on $P$) and every $y\in Y,$
\[
     \eth_{y}u= \lambda u\Longrightarrow \begin{cases}
      P_{y}u=u, & \mbox{if}\; \lambda>R(y), \\
      P_{y}u=0, & \mbox{if} \; \lambda<-R(y).
\end{cases}
\] 
\label{bc.1}\end{definition}

If $P_{1}$ and $P_{2}$ are spectral sections for the family $\eth$, then as
shown in \cite{MR99a:58144}, their formal difference defines a $K$-class
$[P_{1}-P_{2}]\in K^{0}(Y).$ If $P_{1}P_{2}=P_{2},$ then $[P_{1}-P_{2}]$
is represented by the vector bundle given by the range of the family of the
finite rank projections $(\Id-P_{2})P_{1}.$  In general, one can reduce to
this case by choosing a third spectral section $R$ such that $P_{1}R=R,$
$P_{2}R=R$ and setting
\begin{equation}
[P_{1}-P_{2}] = [P_{1}-R]- [P_{2}-R]\in K^{0}(Y).
\label{bc.2}\end{equation}  
It is shown in \cite{MR99a:58144} that such a spectral section $R$
always exists and that the definition of the $K$-class $[P_{1}-P_{2}]$ does
not depend on the choice of $R.$
To obtain a spectral section from an invertible self-adjoint perturbation
$Q$, we need an extra assumption.

\begin{definition} A family $Q\in \Psi^{-\infty}(M/Y;\bbE)$ of self-adjoint
  operators is spectrally finite with respect to the family of Dirac-type
  operators $\eth$ if there exists a smooth function $R:Y\to \bbR^{+}$ such
  that for every $y\in Y$,
\[
\eth_{y}u= \lambda u \Longrightarrow Q_{y}u=0  \mbox{ if }|\lambda|>R(y).
\]
\label{bc.2b}\end{definition}
If $Q\in \Psi^{-\infty}(M/Y;\bbE)$ is an invertible self-adjoint
perturbation of the family $\eth$ which is not spectrally finite, then
using an approximation argument, it is shown in \cite{MR99a:58144} that
it is possible to deform it through invertible self-adjoint perturbations
to one which is spectrally finite.

Suppose now that $Q\in \Psi^{-\infty}(M/Y;\bbE)$ is a spectrally finite
invertible self-adjoint perturbation of the family $\eth.$ Then there is a
corresponding spectral section $P_{Q}\in \Psi^{0}(M/Y;\bbE)$ with
$(P_{Q})_{y}$ defined to be the projection onto the positive eigenspace of
$\eth_{y}+Q_{y}.$ The following relative index theorem was proved in
\cite{MR99a:58144}.

\begin{proposition}\label{bc.3}
If $Q_{1}$ and $Q_{2}$ are two spectrally finite
  invertible self-ajoint perturbations of the family $\eth,$ then the Chern
  character $\Ch([P_{Q_{1}}-P_{Q_{2}}])$ of the $K$-class
  $[P_{Q_{1}}-P_{Q_{2}}]\in K^{0}(Y)$ is represented by the closed form 
\[
\widehat{\eta}_{\BC}(\eth+Q_{1})- \widehat{\eta}_{\BC}(\eth+Q_{2}).
\]
\end{proposition}

If the odd index of the family $\eth$ does not vanish, it is still possible
to define a version of the Bismut-Cheeger eta form, but now over
an infinite dimensional bundle $\varpi:\cA^{\SA}\longrightarrow Y$ defined
in terms of self-adjoint smoothing perturbations. Namely the fibre at
$y\in Y$ is
\begin{equation}
\cA^{\SA}_{y}= \{\eth_{y}+Q_{y}\; ; \; Q_{y}\in \Psi^{-\infty}(Z_{y};\bbE),
\;Q_{y}^{*}=Q_{y}, \; \eth_{y}+Q_{y}\mbox{ is invertible} \}.
\label{bs.6}\end{equation}
The pull-back $\varpi^{*}\cA^{\SA}$ of $\cA^{\SA}$ to itself has a
tautological section $\sigma_{\cA^{\SA}}: \cA^{\SA}\longrightarrow
\varpi^{*}\cA^{\SA}$ which can be used to define a form on the total space
of $\cA^{\SA}$ via formula \eqref{bs.2b},
\begin{equation}
{\eta}_{\BC} \in \Omega^{\even}(\cA^{\SA}).
\label{bs.7}\end{equation}
It is well-defined since the space $\cA^{\SA}$ has a natural structure of smooth Fr\'echet manifold.  As in \eqref{bs.10}, we can also consider its rescaled version
$\widehat{\eta}_{\BC}\in \Omega^{\even}(\cA^{\SA})$.  
Proposition~\ref{bs.3} then extends as follows.

\begin{proposition} The exterior differential of the Bismut-Cheeger eta
form $\eta_{\BC}$ in \eqref{bs.7} is the basic form
\[
d\eta_{\BC}=
\varpi^{*}\left((2\pi i)^{-\frac{n+1}{2}}
\int_{M/Y} \hat{A}(R_{g}) \Ch'(\bbE)\right).
\]
\label{bs.8}\end{proposition}

Let $\rho\in \cS(\bbR)$ be a a choice of Schwartz function such that
$\rho(0)=1$.  Let also $\cA$ be the infinite dimensional bundle of
\eqref{efatopi.82} with $A=\eth$.  The function $\rho$ can be used to
define an injective bundle map $\iota: \cA^{\SA}\to \cA$ defined fibrewise
by
\begin{equation}
\iota_{y}: \cA^{\SA}_{y} \ni \eth_{y}+Q_{y} \longmapsto
\eth_{y}+it+ Q_{y}\rho(t) \in \cA_{y}.
\label{bs.9}\end{equation}
The definition of $\iota$ clearly depends on the choice of $\rho$, but
since the space of Schwartz functions equal to one at the origin is convex,
the homotopy class of the map $\iota$ does not depend on the choice of
$\rho.$

\begin{proposition}\label{bc.4} For each $y\in Y,$ the map $\iota_{y}:
  \cA^{\SA}_{y}\longrightarrow \cA_{y}$ is a weak homotopy equivalence; in
  particular, $\cA^{\SA}_{y}$ is a classifying space for even $K$-theory
  and $\iota: \cA^{\SA}\longrightarrow \cA$ is also an homotopy equivalence.
\end{proposition}

\begin{proof} Let $W\in \Psi^{-\infty}(Z_{y};\bbE_{y})$ be a fixed choice
  of spectrally finite invertible self-adjoint perturbation of $\eth_{y}.$
  Let $B$ be a smooth closed manifold. Given a smooth map $f:B\longrightarrow
  \cA^{\SA}_{y}$ with $f(b)= \eth_{y}+Q_{b},$ it is always possible to
  deform it through self-adjoint invertible perturbations so that the family
  $b\longmapsto Q_{b}\in\Psi^{-\infty}(Z_{y};\bbE_{y})$ becomes spectrally
  finite with respect to $\eth_{y}.$ This means there is a well-defined
  map
\begin{equation}
\mu: [B;\cA^{\SA}_{y}]\ni [f] \longmapsto [P_{f}-P_{W}]\in K^{0}(B) 
\label{bc.4b}\end{equation}
where $P_{f}$ is the spectral section associated to the spectrally finite
invertible self-adjoint perturbation $b\mapsto Q_{b}$ and $P_{W}$ is the
spectral section associated to $W$ seen as a spectrally finite invertible
perturbation for the trivial family $b\mapsto \eth_{y}$ over $B$.  The map
$\mu$ is easily seen to be bijective, so that $\cA^{\SA}_{y}$ is a
classifying space for even $K$-theory.  On the other hand, under the
identification
\begin{equation}
   \cA_{y} \cong G^{-\infty}_{\sus}(Z_{y};\bbE_{y}),
\label{bc.5}\end{equation}
given by composing on the right with $(\iota_{y}(\eth_{y}+W))^{-1}$, we see
the fibre $\cA_{y}$ is also a classifying space for even $K$-theory.  In
fact, this identification induces a map
\[
   \nu: [B;\cA_{y}]\longrightarrow [B;G^{-\infty}(Z_{y};\bbE_{y})]\cong K^{-2}(B)
\]
and a corresponding commutative diagram
\begin{equation}
\xymatrix{  [B;\cA^{\SA}_{y}]
\ar[r]^{\iota_{*}} \ar[d]^{\mu} & [B;\cA_{y}] \ar[d]^{\nu} \\
K^{0}(B) \ar[r]^{p} & K^{-2}(B)
}
\label{bc.6}\end{equation}
where the bottom map $p$ is Bott periodicity. In particular, this
shows the map $\iota:\cA^{\SA}_{y}\longrightarrow \cA_{y}$ is a weak homotopy
equivalence. 
\end{proof}
The proof of the previous proposition also gives the following result.
\begin{corollary}
Suppose that $\sigma_{1}: Y\to \cA^{\SA}$ and $\sigma_{2}: Y\to \cA^{\SA}$ are two sections of the bundle $\cA^{\SA}$. Then the closed forms
\[
    (\sigma_{1}^{*}\widehat{\eta}_{\BC}-\sigma^{*}_{2}\widehat{\eta}_{\BC}) \quad \mbox{and} \quad  ((\iota\circ\sigma_{1})^{*}\eta_{\cA} - (\iota\circ \sigma_{2})\eta_{\cA} )
\] 
represent the same cohomology class in $H^{\even}_{\dR}(Y)$.
\label{bc.6b}\end{corollary}
\begin{proof}
Using section $\sigma_{1}$ instead $W$ to define a map $\mu$ as in \eqref{bc.4b}, we get a commutative diagram  analogous to \eqref{bc.6}, namely, 
\begin{equation}
\xymatrix{  [Y;\cA^{\SA}]
\ar[r]^{\iota_{*}} \ar[d]^{\mu} & [Y;\cA] \ar[d]^{\nu} \\
K^{0}(B) \ar[r]^{p} & K^{-2}(B).
}
\label{bc.6c}\end{equation}
Thanks to Proposition~\ref{bc.3}, we then see that the forms $(\sigma_{1}^{*}\widehat{\eta}_{\BC}-\sigma^{*}_{2}\widehat{\eta}_{\BC})$ and $((\iota\circ\sigma_{1})^{*}\eta_{\cA} - (\iota\circ \sigma_{2})\eta_{\cA} )$ represent the Chern character of the same $K$-class, from which the result follows.
\end{proof}

Since Proposition~\ref{bc.4} shows that the bundles $\cA^{\SA}$ and $\cA$ are homotopy equivalent, the map $\iota$ allows $\eta_{\cA}$ to be compared
with the Bismut-Cheeger eta form $\widehat{\eta}_{BC}.$  The definitions of
$\eta_{\cA}$ and $\widehat{\eta}_{\BC}$ involve different regularizations
of the underlying Dirac family. Assuming choices of regularization and choices of
connections affect the eta form in a similar way, it is natural  in light of
Lemma~\ref{coc.1} to expect the following relation between the two types of eta forms.

\begin{conjecture} There exist forms $\beta\in \Omega^{\even}(Y)$ and $\alpha\in
\Omega^{\odd}(\cA^{\SA})$ such that
\[
\iota^{*}\eta_{\cA}- \widehat{\eta}_{\BC}= \varpi^{*}\beta+ d\alpha.
\]
\label{bs.11}\end{conjecture}
As an indication that this conjecture might be true, we will prove it in a
particular case.

\begin{theorem} The conjecture is true when the odd families index of the
  family $\eth\in \Psi^{1}(M/Y;\bbE)$ vanishes.
\label{bc.7}\end{theorem}
\begin{proof}

Since $\ind(\eth)=0$ in $K^{1}(Y)$, we know $\eth$ admits an invertible
self-adjoint perturbation $\eth+ Q_{0}$ with $Q_{0}\in
\Psi^{-\infty}(M/Y;\bbE).$ Without loss of generality, we can assume
$Q_{0}$ is spectrally finite. The perturbation $Q_{0}$ defines a section
$\sigma:Y\longrightarrow \cA^{\SA}$ with $\sigma(y)= \eth_{y}+(Q_{0})_{y}.$
There is also an induced section $\iota\circ\sigma: Y\longrightarrow \cA$ for the
bundle $\cA\longrightarrow Y.$  The form on the base is then taken to be
\[
\beta= \sigma^{*}\widehat{\eta}_{\BC}- (\iota\circ \sigma)^{*}\eta_{\cA}.
\]

The form $\omega= \widehat{\eta}_{\BC}- \iota^{*}\eta_{\cA}-
\pi^{*}_{\cA^{\SA}}\beta$ can then be written as the difference of two
closed forms, $\omega= \omega_{\BC}-\omega_{\cA}$ with
\[
\omega_{\BC}=
\widehat{\eta}_{\BC}- \pi^{*}_{\cA^{\SA}}\sigma^{*}\widehat{\eta}_{\BC},\
\omega_{\cA}= \iota^{*}\eta_{\cA}- \pi^{*}_{\cA^{\SA}}(\iota\circ\sigma)^{*}\eta_{\cA}.
\]
Let $B$ be a closed manifold and $f: B\longrightarrow \cA^{\SA}$ a smooth map.
By perturbing $f$ in its homotopy class as necessary, it can be arranged that $f$
induces a spectrally finite invertible self-adjoint pertrubation of the
family $f^{*}\eth$ parametrized by $B.$ By Corollary~\ref{bc.6b}, both $f^{*}\omega_{\BC}$ and
$f^{*}\omega_{\cA}$ represents the Chern character of $[P_{f}-
  ((\pi_{\cA^{\SA}}\circ f)^{*}P_{Q}],$ where $P_{f}$ is the spectral
section associated to the invertible family over $B$ defined by the map $f$
and $(\pi_{\cA^{\SA}}\circ f)^{*}P_{Q}$ is the spectral section associated
to the invertible family obtained by pulling back the invertible family
$\eth+Q$ under the map $\pi_{\cA^{\SA}}\circ f: B\longrightarrow Y.$

Since $B$ and $f$ are arbitrary, this means $\omega$ is trivial in the
singular cohomology of $\cA^{\SA}.$  From the infinite dimensional version of
the de Rham theorem in this context (see Lemma~\ref{dR.1} below), it
follows that there exists $\alpha\in \Omega^{\odd}(\cA^{\SA})$ such that
\[
          d\alpha= \omega,
\]
from which the result follows.  
\end{proof}

\begin{lemma}[de Rham theorem]  The de Rham theorem holds for the infinite dimensional space $\cA^{\SA}$.
\label{dR.1}\end{lemma}
\begin{proof}
According to Theorem 16.10 and 34.7 in \cite{Kriegl-Michor} or the discussion on p.25 of \cite{Pressley-Segal1}, it suffices to show that $\cA^{\SA}$ satisfies the following two properties:
\begin{itemize}
\item[(i)] the Fr\'echet space $\cF$ on which $\cA^{\SA}$ is locally modelled has enough smooth functions, which means that for each open set $\cU$ in $\cF$, there is a nonvanishing real-valued smooth function which vanishes outside $\cU$;
\item[(ii)] The manifold $\cA^{\SA}$ is Lindel\"of, which means each open covering has a countable refinement.   
\end{itemize}
To show property (i), fix $y\in Y$ and consider the Fr\'echet space $\Psi^{-\infty}(Z_{y};\bbE_{y})$.  Then the closed subspace 
\[
   \Psi^{-\infty}_{\SA}(Z_{y};\bbE_{y})=\{ Q\in \Psi^{-\infty}(Z_{y};\bbE_{y})\; ; \;Q^{*}=Q\} 
\]
is also naturally a Fr\'echet space.  If $k=\dim Y$, then our local model for $\cA^{\SA}$ can be taken to be the Fr\'echet space
\[
     \cF= \Psi^{-\infty}_{\SA}(Z_{y};\bbE_{y})\times \bbR^{k}.
\]
Since $\Psi^{-\infty}(Z_{y};\bbE_{y})$ is a nuclear Fr\'echet space, so is $\cF$ (see Corollary~21.6.4 and Corollary~21.2.3 in \cite{Jarchow}).  Thus, by Proposition~14.4 in \cite{Kriegl-Michor}, $\cF$ has enough smooth functions.  

To prove property (ii), notice that the Fr\'echet space $\Psi^{-\infty}(Z_{y};E_{y})$ is separable, so it is in particular second-countable.  This implies $\Psi^{-\infty}_{\SA}(Z_{y};\bbE_{y})$ and $\cA^{\SA}_{y}$ are second-countable, and more generally, that $\cA^{\SA}$ is second-countable, which means in particular it is Lindel\"of.

\end{proof}


\providecommand{\bysame}{\leavevmode\hbox to3em{\hrulefill}\thinspace}
\providecommand{\MR}{\relax\ifhmode\unskip\space\fi MR }
\providecommand{\MRhref}[2]{%
  \href{http://www.ams.org/mathscinet-getitem?mr=#1}{#2}
}
\providecommand{\href}[2]{#2}

\end{document}